\documentclass[10pt]{article} 
\usepackage{tikz}
\usepackage{tikz-cd}
\usepackage{xparse}
\usepackage{pgfplots}
\usetikzlibrary{matrix}  
\usetikzlibrary{decorations.markings}
\usepackage{boxedminipage}
\usepackage{scalerel,amsthm}
\usepackage{mathtools}
\usepackage{subfig}
\usepackage[lite]{amsrefs} %for references
\usepackage{amsmath,amssymb}% for rtimes
\usepackage[all]{xy} %for diagrams
\usepackage{enumerate} %for options in enumerate environments
\usepackage{enumitem}
\usepackage{hyperref}
\usepackage{graphicx} %for images
\usetikzlibrary{matrix}
\usepackage{xcolor}
\numberwithin{equation}{section}
\usetikzlibrary{backgrounds}
\newtheorem{THM}{Theorem}

\usepackage[customcolors]{hf-tikz}
\usetikzlibrary{matrix,decorations.pathreplacing}
\pgfkeys{tikz/mymatrixenv/.style={decoration=brace,every left delimiter/.style={xshift=3pt},every right delimiter/.style={xshift=-3pt}}}
\pgfkeys{tikz/mymatrix/.style={matrix of math nodes,left delimiter=[,right delimiter={]},inner sep=2pt,column sep=1em,row sep=0.5em,nodes={inner sep=0pt}}}
\pgfkeys{tikz/mymatrixbrace/.style={decorate,thick}}

\newcommand\mydots{\hbox to 1em{.\hss.\hss.}}
\makeatletter
\newcommand*{\rom}[1]{\expandafter\@slowromancap\romannumeral #1@}
\makeatother

\newcommand{\subscript}[2]{$#1 _ #2$}

\newtheorem{theorem}{Theorem}[section]

\newtheorem{lemma}[theorem]{Lemma}
\newtheorem{proposition}[theorem]{Proposition}
\newtheorem{corollary}[theorem]{Corollary}
\theoremstyle{definition}
\newtheorem{example}[theorem]{Example}
\newtheorem{dfn}[theorem]{Definition}
\newtheorem{remark}[theorem]{Remark}

\ExplSyntaxOn
\NewDocumentCommand{\xhrulefill}{O{}}
 {
  \group_begin:
  \severin_xhrulefill:n { #1 }
  \group_end:
 }

\keys_define:nn { severin/xhrulefill }
 {
  height .dim_set:N    = \l_severin_xhrule_height_dim,
  thickness .dim_set:N = \l_severin_xhrule_thickness_dim,
  fill .skip_set:N     = \l_severin_xhrule_fill_skip,
  height .initial:n    = 0pt,
  thickness .initial:n = 0.4pt,
  fill .initial:n      = 0pt plus 1fill,
 }

\cs_new_protected:Nn \severin_xhrulefill:n
 {
  \keys_set:nn { severin/xhrulefill } { #1 }
  \leavevmode
  \leaders\hrule 
    height \dim_eval:n { \l_severin_xhrule_thickness_dim + \l_severin_xhrule_height_dim }
    depth  \dim_eval:n { -\l_severin_xhrule_height_dim }
  \skip_horizontal:N \l_severin_xhrule_fill_skip
  \kern 0pt
}
\ExplSyntaxOff

\DeclareMathOperator{\im}{ran}

\newcommand{\unit}{1\!\!1}

\DeclareMathOperator{\vndet}{det_{\Gamma}}

\DeclareMathOperator{\vntr}{tr_{\Gamma}}

\DeclareMathOperator{\vnH}{\mathcal H}

\DeclareMathOperator{\vndim}{\dim_{\mathcal N(\Gamma)}}

\DeclareMathOperator{\GL}{GL}

\DeclareMathOperator{\reals}{\mathbb R}

\DeclareMathOperator{\ceals}{\mathbb C}

\DeclareMathOperator{\dom}{dom}
\DeclareMathOperator{\Vol}{Vol}

\DeclareMathOperator{\hyp}{\mathbb H}

\DeclareMathOperator{\Wh}{Wh}

\DeclareMathOperator{\tr}{tr}
\DeclareMathOperator{\Isom}{Isom}

\DeclareMathOperator{\bundf}{\xrightarrow{F}}
\DeclareMathOperator{\vnN}{\mathcal N}

\DeclareMathOperator{\sob}{\mathcal  W}

\newcommand{\heatR}{e^{-t\Delta_p[E^\rho_{M_R}]}}
\newcommand{\heatH}{e^{-t\Delta_p[E^\rho]}}
 
\DeclareMathOperator{\clos}{clos}
{}

\DeclareMathOperator{\SO}{\mathrm{SO}}

\DeclareMathOperator{\zvndet}{det_{\Gamma}^{\zeta}}

\usetikzlibrary{arrows.meta,bending}

\title{The $L^2$-torsion for representations of hyperbolic lattices} 
\author{Benjamin Wa{\ss}ermann}

\begin{document}

\maketitle
\tikzset{middlearrow/.style={
        decoration={markings,
            mark= at position 0.5 with {\arrow{#1}} ,
        },
        postaction={decorate}
    }
}

\abstract{We prove equality of analytic and topological $L^2$-torsion associated with an odd-dimensional finite volume hyperbolic manifold and a representation of the fundamental group which extends to the ambient Lie group. This generalizes a previous result due to L\"uck and Schick. Alternatively, this result can be regarded as the $L^2$-analogue of recent work by M\"uller and Rochon. } \\

\section{Introduction}
For $n$ an odd integer, let $M$ be a complete, oriented, finite volume hyperbolic $n$-manifold with empty boundary and fundamental group $\Gamma$. Then $\Gamma$ can be identified with a lattice inside $G \coloneqq \SO_0(n,1)$. Moreover, with $K \coloneqq \SO(n) < G$, we obtain $n$-dimensional hyperbolic space $\hyp^n \cong G/K$ and recover $M$ as the quotient $\Gamma \backslash \hyp^n$. \\
We also fix a complex, finite-dimensional, irreducible representation $\rho$ of $G$ and let $E^\rho \downarrow \hyp^n$ be the $G$-equivariant flat bundle over $\hyp^n$ associated with $\rho$. 
By \cite[Lemma 3.1]{Matsushima:Metric}, one can equip $E^\rho$ with a $G$-equivariant Hermitian metric $h_\rho$, unique up to a positive scalar. In \cite[Section $9$]{Muller:heat}, M\"uller and Pfaff construct from $E^\rho$ and $h_\rho$ the analytic $L^2$-torsion \begin{equation*} T_{(2)}^{An}(M,\rho) \in \reals_{>0}, \end{equation*} also providing some computational tools along the way with aid of the {\itshape Plancherel formula}, at least under the additional assumption that $\rho$ has non-zero highest weight \cite[Equation $(9.5)$]{Muller:heat} (for the formal definition of $T_{(2)}^{An}(M,\rho)$, this assumption is not necessary, as shown in \cite[Corollay $1.3$]{Durland:Novi}). \\
Rather than delving into more explicit computations of $T_{(2)}^{An}(M,\rho)$, writing this paper was inspired by $L^2$-Cheeger-M\"uller type results (to be listed below). Correspondingly, we are going to verify that $T_{(2)}^{An}(M,\rho)$ has a counterpart in the form of the topological $L^2$-torsion \begin{equation*} T^{Top}_{(2)}(M,\rho) \in \reals_{>0}. \end{equation*} For the definition of $T^{Top}_{(2)}(M,\rho)$, a few remarks are in order: Firstly, $T^{Top}_{(2)}(M,\rho)$ is constructed from a choice of finite $\Gamma$-CW complex $X$ which is required to be $\Gamma$-homotopy equivalent to $\hyp^n$. Whenever $\Gamma$ acts cocompactly on $\hyp^n$, one may simply choose $X$ to be a CW-structure of $\hyp^n$ lifted from the compact quotient $\Gamma \backslash \hyp^n$. In general, one may always choose $X$ to be a lift of a CW-structure of $\overline{\Gamma \backslash \hyp^n}$, where $\overline{\Gamma \backslash \hyp^n}$ is a compact manifold with boundary whose interior is diffeomorphic to $\Gamma \backslash \hyp^n$, and whose construction is laid out in \cite[Chapter $4$]{Petronio:hyp} . \\
Secondly, it is crucial that $T^{Top}_{(2)}(M,\rho)$ should only depend on $M$ and $\rho$, but {\bfseries not} on the specific choice of $X$. As carefully explained in \cite[Section $5.3$]{Ich}, this follows from the well-known facts that the representation $\rho$ is unimodular \cite[Lemma $4.3$]{Muller:Tor2} and that the {\itshape Whitehead group} $\Wh(\Gamma)$ of $\Gamma$ vanishes \cite[Lemma $0.12$]{Farrell:Wh}. \\
Thirdly, and perhaps most importantly, the pair $(M,\rho)$ needs to be {\itshape $L^2$-acyclic}, i.e.\ have vanishing topological $L^2$-Betti numbers, and satisfy a {\itshape determinant class} condition. The verification of these two properties will be a side product of our investigation, summarized in Section \ref{FLATRHO} .
The main result of this paper can now be formulated: \begin{THM}[Theorem \ref{HypCheMül}]\label{A} One has \begin{equation} T^{Top}_{(2)}(M,\rho) = T_{(2)}^{An}(M,\rho). \end{equation} \end{THM} There are various similar results in the literature, starting with the work of Friedlander, Kappeler, Burghelea and MacDonald, who analyzed the difference of analytic and topological $L^2$-Theorem in two consecutive papers. In \cite{Friedlander:Uni}, they considered arbitrary closed (i.e.\ compact and without boundary) Riemannian manifolds $(M,g)$ and arbitrary {\itshape unitary} $L^2$-acyclic representations $\rho$ of their fundamental group, proving equality of the resulting analytic and topological $L^2$-torsions. In \cite{Friedlander:Bd}, they extended their previous result onto compact manifolds with boundary $(M,g)$, showing that, under the assumption that $g$ is a product near $\partial M$, the $L^2$-torsions are equal up to a constant depending only on $\dim(\rho)$ and $\chi(\partial M)$. Later, Zhang \cite{Zhang:CM} generalized the result of \cite{Friedlander:Uni} for $(M,g)$ still closed, but now $\rho$ being an arbitrary {\itshape unimodular} representation. Finally, the author of this paper generalized both results in \cite{Ich2}, in which he studied compact manifolds with boundary $(M,g)$ with no assumptions on $g$ and arbitrary unimodular representations $\rho$. In fact, the main result of that paper will be put to use in this work. \\ Note, however, that, unlike Theorem \ref{A}, none of the aforementioned results deal with the case when $(M,g)$ is non-compact. One work from the established literature dealing with the non-compact case, and also serving as the main inspiration to Theorem \ref{A} is the classic paper by L\"uck and Schick \cite{Lueck:hyp}, who showed that $T^{Top}_{(2)}(M,\unit) = T_{(2)}^{An}(M,\unit)$ with $\unit \colon G \to \mathbb C^\times$ for $M$ being a non-compact odd-dimensional complete hyperbolic manifold with finite volume and $\unit$ denoting the trivial representation. In fact, we will adapt the main idea of their proof, at the heart of which lies considering an exhaustion $(M_R)_{R \in \reals_{>0}}$ of $\hyp^n$ of $\Gamma$-invariant complete submanifolds with boundary, such that each quotient $\Gamma \backslash M_R \subseteq M$ is compact and homotopy equivalent to $M$. The majority of this paper is devoted to the proof of the following limit formula: 
\begin{THM}[Theorem \ref{anconv}]\label{B} One has \begin{equation} \lim_{R \to \infty} T_{(2)}^{An}(\Gamma \backslash M_R,\rho) = T_{(2)}^{An}(M,\rho). \end{equation} \end{THM}
Since $T^{Top}_{(2)}(\Gamma \backslash M_R,\rho) = T^{Top}_{(2)}(M,\rho)$ due to both spaces having the same (simple) homotopy type, Theorem \ref{A} would directly follow from Theorem \ref{B} if one could show that the {\itshape torsion anomaly} \begin{equation} \log \left( \frac{T_{(2)}^{An}(\Gamma \backslash M_R,\rho)}{T^{Top}_{(2)}(\Gamma \backslash M_R,\rho)} \right) \end{equation} vanishes in the limit $R \to \infty$. For $\rho = \unit$, L\"uck and Schick were able to show that this is indeed the case \cite[Section $1$]{Lueck:hyp}. However, what was lacking back at the time of their writing was any understanding of the torsion anomaly on compact manifolds with boundary for arbitrary unimodular representations. As mentioned before, the author has undergone the investigation of this general anomaly in a separate work \cite[Theorem $1.2$]{Ich2}, from which the next result can be derived:
\begin{THM}[Theorem \ref{TorQuot}]\label{C} One has \begin{equation} \log \left( \frac{T_{(2)}^{An}(\Gamma \backslash M_R,\rho)}{T^{Top}_{(2)}(\Gamma \backslash M_R,\rho)} \right) =  \dim(\rho) \cdot \log \left( \frac{T_{(2)}^{An}(\Gamma \backslash M_R,\unit)}{T^{Top}_{(2)}(\Gamma \backslash M_R,\unit)} \right). \end{equation} \end{THM} Coincidentally, M\"uller and Rochon recently studied the non-$L^2$ version of Theorem \ref{A} for representations $\rho$ of $G$ that are not fixed by the standard Cartan involution on $G$ corresponding to $K < G$ \cite{Muller:Hyp}. In contrast to the $L^2$-case, they showed that the two resulting torsion elements differ by a term related to the cusps of $\Gamma \backslash \hyp^n = M$ which contribute non-trivially to the twisted cohomology $H^*(M,\Gamma \backslash E^\rho)$ \cite[Theorem $1.1$]{Muller:Hyp}. 
In fact, the situation they face (and, consequently, the method they employ) is fundamentally different from ours, the main difference being that the bundle $E^\rho \downarrow \hyp^n$ central to our investigation is homogeneous (cf.\ Definition \ref{HOMDEF}), unlike the quotient bundle $\Gamma \backslash E^\rho \downarrow \Gamma \backslash \hyp^n$ they consider. \\
Another major inspiration for the development of this work is the influential paper of Bergeron and Venkatesh \cite{Bergeron:Growth}, which has sparked and renewed the general interest in gaining a better understanding of $L^2$-torsion. Providing a computational approach for the analytic $L^2$-torsion of arbitrary non-positively curved, locally symmetric spaces $(M,g)$ and representations $\rho$ of the ambient Lie group \cite[Proposition $5.2$]{Bergeron:Growth}, they also formulated a wide-reaching conjecture with ramifications into the field of arithmetic geometry \cite[Conjecture $1.3$]{Bergeron:Growth}. Namely, they proposed that, whenever the fundamental group $\Gamma$ is a {\itshape congruence} lattice, $T^{An}_{(2)}(M,\rho)$ should be a predictor for the growth rate of the torsion part of certain integral homology groups that are naturally derived from a residually finite sequence $\dots < \Gamma_3 < \Gamma_2 < \Gamma$ of congruence subgroups and the representation $\rho$. This growth rate conjecture is not only of great interest \cite{Ash:Tor}\cite{Raim:Tor}\cite{Thang:Tor} to topologists and geometers - elements inside the aforementioned torsion homology groups that also survive in some $\bmod  \: p$ homology are of great significance in number theory, see \cite{Scholze:Tor} for more details. For a comprehensive introduction to the Bergeron-Venkatesh conjecture, we refer to \cite[Section 6.6]{Kammeyer:L2}. \\
At first glance, it might seem peculiar why the conjecture includes $T^{An}_{(2)}(M,\rho)$ instead of $T^{Top}_{(2)}(M,\rho)$, since the construction of the twisted homology groups is inherently topological in nature. However, $T^{An}_{(2)}(M,\rho)$ is much simpler to define in the locally symmetric setting. Namely, if $M$ is locally symmetric, $T^{An}_{(2)}(M,\rho)$ is always constructed from a homogeneous bundle $E^\rho \downarrow \widetilde{M}$ over the globally symmetric universal cover $\widetilde{M}$, which depends only on the global representation $\rho$, but not on $M$ itself. In the case that the {\itshape fundamental rank} of the ambient Lie group $G$ is one, this further leads to the computational formulas of $T^{An}_{(2)}(M,\rho)$ which take only the volume of $M$ into account \cite[Proposition $5.2$]{Bergeron:Growth}. Conversely, if $M$ is not compact, it is as of now still unclear if $T^{Top}_{(2)}(M,\rho)$ can be defined in general, the problem being that the $L^2$-acyclicity and {\itshape determinant class} conditions necessary for its construction can {\itshape a priori} not be deduced from the analytic bundle, where they are known to be always satisfied. This is because, at the time of writing this paper, any de Rham-type theorems that effectively allow comparing analytic and topological $L^2$-invariants (just like Theorem \ref{DeRh}) are only avaible for compact manifolds. Without these tools at hand, the $L^2$-acyclicity and determinant class conditions can only be verified from a particular good choice of finite CW-model $X$ of $M$, or/and from a better general understanding of the Novikov-Shubin invariants of matrices over the {\itshape complex} group ring $\mathbb C[\Gamma]$. \\ 
In this vein, the result of our paper raises the question whether $T^{Top}_{(2)}(M,\rho)$ can also be defined for arbitrary non-positively curved, locally symmetric spaces, and whether it always equals $T^{An}_{(2)}(M,\rho)$. However, while Theorem \ref{C} can potentially still be adapted, the proof of Theorem \ref{B} doesn't admit a straightforward generalization, since it relies on the well understood (and specific) end structure of hyperbolic manifolds. For example, we will make use of the fact that each quotient $\Gamma \backslash M_R$ only has toroidal boundary components. \\
In Section \ref{Prelim}, we are going to review the basic notions of Hilbert $\vnN(\Gamma)$-modules and cochain complexes, used then to define analytic and topological $L^2$-Betti numbers, Novikov-Shubin invariants and $L^2$-torsion of manifolds and representations of their fundamental group. \\
Section \ref{RelL2} is comprised of previously known results which will be used in subsequent sections, namely the induction principle, the $L^2$ de Rham theorem, as well as the $L^2$-Cheeger M\"uller type theorem for compact manifolds with boundary.  \\
In Section \ref{FLATRHO}, we introduce the homogeneous bundle $E^\rho \downarrow \hyp^n$ associated with a representation $\rho$ of $G \coloneqq SO_0(n,1)$, show that both $T^{An}_{(2)}(\Gamma \backslash M,\rho)$ and $T^{Top}_{(2)}(\Gamma \backslash M,\rho)$ are well-defined for any lattice $\Gamma < G$, and explain how Theorem \ref{A} follows from the limit formula of Theorem \ref{B} and Theorem \ref{C}. Finally, the last two sections are entirely devoted to the proof of Theorem \ref{B}. \\
The paper is based on parts of the author's dissertation. I thank my PhD advisor, Roman Sauer, for his constant support and numerous fruitful discussions. Furthermore, I am indebted to my colleague, Holger Kammeyer, for his many helpful ideas and comments. \\This work is generously funded by the DFG (Deutsche Forschungsgemeinschaft) -- 28186985.

\section{Preliminaries} \label{Prelim}
We start by giving some basic introduction to the general theory of Hilbert $\vnN(\Gamma)$-modules, with the aid of which we further define the relevant analytic and topological $L^2$-invariants. The reader well-versed in the construction of $L^2$-Betti numbers and $L^2$-torsion, both topological and analytical, may skip most parts of this section, all expect for Subsection \ref{1.3}, which contains the statement of a central result used throughout the main proofs of this paper. 

\subsection{Hilbert $\vnN(\Gamma)$-modules}\label{1.1}
Throughout this section, we fix a countable group $\Gamma$. We define $L^2(\Gamma)$ to be the complex Hilbert space generated by the set $\Gamma$. Note that left multiplication by group elements naturally induces on $L^2(\Gamma)$ a linear isometric left action by $\Gamma$. Also, observe that for any complex Hilbert space $H$, this action by $\Gamma$ on $L^2(\Gamma)$ extends naturally to a linear isometric action on the tensor product of Hilbert spaces $L^2(\Gamma)  \otimes H$. A complex Hilbert space $\vnH$, equipped with a linear isometric left $\Gamma$-action and which admits a linear, isometric and $\Gamma$-equivariant embedding $\iota \colon \vnH \hookrightarrow L^2(\Gamma)  \otimes H$ for some separable Hilbert space $H$ is called a {\itshape Hilbert $\vnN(\Gamma)$-module}. $\vnH$ is said to be {\itshape finitely generated} if one can choose $H$ to satisfy $\dim_{\ceals}(H) < \infty$ in the above embedding. \\
Given two Hilbert $\vnN(\Gamma)$-modules $\vnH$ and $\vnH'$, an unbounded (i.e.\ not necessarily bounded), $\Gamma$-equivariant, closed and densely defined operator $f \colon \vnH \to \vnH'$ is called a {\itshape morphism of Hilbert $\vnN(\Gamma)$-modules}. Until the end of this section, due to lack of ambiguity, we will simply refer to any such map as a {\itshape morphism}. \\
The subclass of {\itshape bounded} morphisms is closed under taking arbitrary compositions. In general,
if $f \colon \vnH \to \vnH'$ is a morphism, then so is its orthogonal complement $f^\perp \colon \ker(f)^\perp \to \overline {\im(f)}$, its adjoint $f^* \colon \vnH \to \vnH'$, as well as the self-adjoint composition $f^*f \colon \vnH \to \vnH$ \cite[Theorem $3.24$]{Kato:Func}. Moreover, for any bounded Borel function $\phi \colon \reals \to \ceals$, the operator $\phi(f^*f) \colon \vnH \to \vnH$ constructed via the spectral theorem will be a bounded endomorphism. Finally, if  $f = u |f| \colon \vnH \to \vnH'$ is the (left) polar decomposition of a morphism $f$, both the partial isometry $u \colon \ker(f)^\perp \to \vnH'$, as well as absolute value $|f| \colon \vnH \to \vnH$ are again morphisms. \\ Any {\itshape positive}, bounded endomorphism $f \colon \vnH \to \vnH$ admits the so-called {\itshape von Neumann trace} \begin{equation} \vntr(f) \in [0,\infty], \end{equation} see \cite[Section 1.1]{Lueck:Book} for a comprehensive definition and introduction. A bounded morphism $f \colon \vnH \to \vnH'$ is said to be {\itshape of $\Gamma$-trace class} if $\vntr(|f|) < \infty$. Just as in the trivial case $\Gamma = \{0\}$ (in which case Hilbert $\vnN(\Gamma)$-modules are simply Hilbert spaces and morphisms between them are simply unbounded linear operators), one checks that inside the algebra of bounded endomorphisms over a fixed Hilbert $\vnN(\Gamma)$-module, the trace class operators form a two-sided ideal. A Hilbert $\vnN(\Gamma)$-module is said to be {\itshape finite-dimensional} if the identity morphism $\unit_{\vnH} \colon \vnH \to \vnH $ is of trace class. In this case, we define the {\itshape von Neumann dimension} of $\vnH$ as\begin{equation} \vndim(\vnH) \coloneqq \vntr(\unit_{\vnH}) \in \reals_{\geq 0}. \end{equation} Any finitely-generated Hilbert $\vnN(\Gamma)$-module is also finite-dimensional, but not necessarily vice versa. \\ For $\lambda \geq 0$, let $\chi_{[0,\lambda]} \colon \reals \to \{0,1\}$ be the characteristic function of the subset $[0,\lambda] \subset \reals$, and, for a morphism $f \colon \vnH \to \vnH'$, let $\chi_{[0,\lambda]}(|f|) \colon \vnH \to \vnH$ be the corresponding positive, bounded endomorphism obtained via applying the spectral theorem to the self-adjoint absolute value $|f|$. Define the {\itshape spectral density function} of $f$ as \begin{equation} F(f,\lambda) \coloneqq \vntr( \chi_{[0,\lambda]}(|f|) ) \in [0,\infty]. \end{equation} A morphism $f$ is {\itshape Fredholm} if for each $\lambda \geq 0$, the operator $\chi_{[0,\lambda]}(|f|)$ is of $\Gamma$-trace class, i.e.\ if $F(f,\lambda) < \infty$ for each $\lambda \geq 0$. Notice that morphisms between finite-dimensional Hilbert $\vnN(\Gamma)$-modules are automatically Fredholm.\\
$f$ is said to be {\itshape $\zeta$-regular} if \begin{enumerate}[label=(\alph*)] \item for each $t > 0$, the operator $e^{-t|f|^\perp}$ is of $\Gamma$-trace class, and \item there exists $C > 0$, such that the formal expression \begin{equation} \zeta_{f}(s) \coloneqq \Gamma(s)^{-1}\int_{0}^{1} t^{s-1} \vntr(e^{-t|f|^\perp}) dt \end{equation} determines a holomorphic function on the half-plane $\{s \in \ceals \colon \Re(s) > C \}$ that extends to a meromorphic function on all of $\ceals$ which is regular at $s = 0$. 
\end{enumerate}
Observe that $\chi_{[0,\lambda]}(|f|^\perp) = \chi_{[0,\lambda]}(|f|^\perp) \cdot e^{|f|^\perp} \cdot \chi_{[0,\lambda]}(|f|^\perp) \cdot e^{-|f|^\perp}$ for each $\lambda > 0$. Together with the fact that $\Gamma$-trace class morphisms form an ideal inside the algebra of bounded endomorphisms over a Hilbert $\vnN(\Gamma)$-module, it follows that a $\zeta$-regular morphism with finite-dimensional kernel must be Fredholm as well. \\
A Fredholm morphism is said to be {\itshape of determinant class} if the integral $\int_{0^+}^1 \ln(\lambda) dF(f,\lambda)$ converges. Here $dF(f,\lambda)$ denotes the Borel measure on $\reals_{\geq 0}$ induced by the positive, right-continuous function $F(f,\lambda)$. A $\zeta$-regular morphism $f$ is of determinant class if and only if the integral $\int_{1}^\infty t^{-1} \vntr(e^{-t|f|^\perp})$ converges \cite[Lemma $3.139$]{Lueck:Book}. As such, we may define for any such $f$ the {\itshape $\zeta$-regularized determinant}
\begin{equation} \zvndet(f) \coloneqq \exp \left( - \frac{d}{ds} \zeta_{f}(s)|_{s=0} - \int_{1}^\infty t^{-1} \vntr(e^{-t|f|^\perp}) dt \right) \in \reals_{>0}. \end{equation}
As we shall see soon, $\zvndet$ can be defined for many different unbounded endomorphisms over infinitely-generated Hilbert $\vnN(\Gamma)$-modules. Moreover, in some instances, $\zvndet$ can be regarded as an extension of the {\itshape Fuglede-Kadision determinant} $\vndet$ \cite[Definition $3.11$]{Lueck:Book}, which we won't discuss here in any more detail and which is defined only for bounded endomorphisms over finitely-generated Hilbert $\vnN(\Gamma)$-modules. Indeed, these two notions of determinant coincide over the class of bounded endomorphisms of finitely-generated Hilbert $\vnN(\Gamma)$-modules with a spectral gap at $0$ \cite[Proposition $4.1.25$]{Ich}. \\ Finally, as a quantification for the spectral behavior of a morphism $f$ near $0$ with the property that $f^\perp$ is Fredholm, the {\itshape Novikov-Shubin Invariant} $\alpha(f) \in [0,\infty] \cup \{\infty^+\}$ of $f$ is defined as 
\begin{equation} \alpha(f) \coloneqq \begin{cases} \liminf_{\lambda \to 0^+} \cfrac {\ln(F(f^\perp,\lambda))}{\ln(\lambda)} \in [0,\infty],& \text{if} \; F(f^\perp,\lambda) > 0 \;\; \forall \lambda > 0,\\ \infty^+, & \text{else}. \end{cases} \end{equation} Observe that $\alpha(f)$ equals the (purely formal) symbol $\infty^+$ if and only if $f$ has a spectral gap at $0$. Moreover, $\alpha(f) > 0$ implies that $f$ is of determinant class (the converse need not hold). \\
%Furthermore, if $f$ is additionally $\zeta$-regular, one can show that \begin{equation} \alpha(f) \coloneqq \begin{cases} \sup \{ \alpha \in [0,\infty] \colon \vntr(e^{-t|f|^\perp}) \in \mathcal O(t^{-\alpha}) \; \text{as} \; t \to \infty \}, & \text{if} \; F(f^\perp,\lambda) > 0 \;\; \forall \lambda > 0,\\ \infty^+, & \text{else}, \end{cases} \end{equation} from which follows that a $\zeta$-regular morphism satisfying $\alpha(f) \neq 0$ is of determinant class. \\ \newline
A diagram \begin{equation} (C_*,c_*) \colon 0 \xrightarrow{0} C_0 \xrightarrow{c_0} C_1 \xrightarrow{c_1} C_2 \xrightarrow{c_2} \dots, \end{equation} where each $C_n$ is a Hilbert $\vnN(\Gamma)$-module with $C_n = 0$ for all but finitely many $n \in \mathbb N_0$ and each $c_n$ a morphism satisfying $\overline{\im(c_n)} \subseteq \ker(c_{n+1})$ is called a {\itshape Hilbert $\vnN(\Gamma)$-cochain complex}. For the remainder of this section, we will simply refer to these type of complexes as {\itshape cochain complexes}. Such a complex is said to be {\itshape of finite type} if each $C_n$ is finitely-generated and each $c_n$ is bounded. $C_*$ is said to be {\itshape Fredholm at $n \in \mathbb N_0$} if the restricted boundary operator $c_n|_{\im(c_{n-1})^\perp}$ is Fredholm, and {\itshape Fredholm} if it is Fredholm at each $n \in \mathbb N_0$. In case that $C_*$ is Fredholm, we can define for each $n \in \mathbb N_0$ its $n$-th {\itshape spectral density function} \begin{align} F_n(C_*,  \_\_ ) \colon \reals_{\geq 0} \to \reals_{\geq 0}, \hspace{.5cm} F_n(C_*,\lambda) \coloneqq F_n(c_n|_{\im(c_{n-1})^\perp},\lambda), \end{align} its $n$-th {\itshape $L^2$-Betti Number} \begin{align} b_n^{(2)}(C_*) \coloneqq F_n(C_*,0) \in \reals_{\geq 0}, \end{align} as well as its $n$-th {\itshape Novikov-Shubin invariant} \begin{equation} \alpha_n(C_*) \coloneqq \alpha(c_n|_{\im(c_{n-1})^\perp}) \in [0,\infty] \cup \{\infty^+\}. \end{equation} A Fredholm complex is said to be {\itshape acyclic} if $b_n^{(2)}(C_*) = 0$ for each $n \in \mathbb N_0$. \\ %Both the $L^2$-Betti numbers, as well as the Novikov-Shubin invariants of complexes are  
If $(C_*,c_*)$ and $(D_*,d_*)$ are two cochain complexes, a sequence $f_* \colon C_* \to D_* \coloneqq (f_n \colon C_n \to D_n)_{n \in \mathbb N_0} $ of {\itshape bounded} morphisms that satisfies both $f_n(\dom(c_n)) \subseteq \dom(d_n)$ and $d_n \circ f_n = f_{n+1} \circ c_n$ for all $n \in \mathbb N_0$ is called a {\itshape morphism of (Hilbert $\vnN(\Gamma)$-)cochain complexes}.  Such $f_*$ is called an {\itshape isomorphism} if each $f_n \colon C_n \to D_n$ is an isomorphism of Hilbert $\vnN(\Gamma)$-modules. \\
Two morphisms $f_* \colon C_* \to D_*$ and $g_* \colon C_* \to D_*$ of cochain complexes are {\itshape chain homotopic} if there exists a chain homotopy $K_* \colon C_* \to D_{*-1}$ between them, i.e.\ a sequence of bounded morphisms $K_n \colon C_n \to D_{n-1}$ for each $n \in \mathbb N$ satisfying $K_{n+1} \circ c_n + d_{n-1} \circ K_n = f_n - g_n$. Two complexes $C_*$ and $D_*$ are {\itshape chain homotopy equivalent} if there exists morphisms $f_* \colon C_* \to D_*$ and $g_* \colon D_* \to C_*$, such that $g_* \circ f_*$ is chain homotopic to the identity morphism $\unit_{C_*}$ and $f_* \circ g_*$ is chain homotopic to $\unit_{D_*}$.
\begin{theorem}\cite[Proposition $4.1$]{Gromov:Novi}\label{DILEQ} Let $C_*$ and $D_*$ be two $\zeta$-regular Fredholm complexes that are chain homotopy equivalent. Then the spectral density functions are dilatationally equivalent near $0$, that is, there exist constants $c,\epsilon > 0$  so that we have for each $n \in \mathbb N_0$ and each $\lambda \leq \epsilon$
\begin{equation} F_n(C_*,c^{-1} \cdot \lambda) \leq F_n(D_*,\lambda)  \leq F_n(C_*,c \cdot \lambda). \end{equation}
Consequently, we get \begin{align} & b_n^{(2)}(C_*) = b_n^{(2)}(D_*), \\ & \alpha_n(C_*) = \alpha_n(D_*) \end{align} for each $n \in \mathbb N_0$. In particular, $C_*$ is acyclic if and only if $D_*$ is acyclic. \end{theorem}
A diagram of Hilbert $\vnN(\Gamma)$-cochain complexes \begin{equation} 0 \rightarrow C_* \xrightarrow{f_*} D_* \xrightarrow{g_*} E_* \rightarrow 0 \end{equation} is called a {\itshape short exact sequence} if $f_n: C_n \to D_n$ is injective, $g_n: D_n \to E_n$ is surjective and $\ker(g_n) = \im(f_n)$ for each $n \in \mathbb N_0$. A fundamental result that relates the spectral density functions of cochain complexes in a given short exact sequence is the following:
\begin{proposition}\cite[Theorem 4.11]{Lueck:hyp}\label{exactleq} Let  $(C_*,c_*)$, $(D_*,d_*)$ and $(E_*,e_*)$ be three  Hilbert $\vnN(\Gamma)$-cochain complexes, and \begin{equation} 0 \rightarrow C_* \xrightarrow{f_*} D_* \xrightarrow{g_*} E_* \rightarrow 0 \end{equation} a short exact sequence of morphisms between them. 
Suppose that $C_*$ and $E_*$ are both Fredholm at $n \in \mathbb N_0$ and that either $b_n(E_*) = 0$ or $b_{n+1}(C_*) = 0$. Then $D_*$ is also Fredholm at $n$ and we have the estimates 
\begin{equation} \hat{F}_n(D_*,\lambda) \leq \hat{F}_n(C_*,\alpha_C \cdot \lambda^{1/2}) + \hat{F}_n(E_*, \alpha_E \cdot \lambda^{1/2}) \; \; \text{for} \; \; 0 \leq \lambda < \alpha_1, \end{equation} where
\begin{align} &\alpha_C \coloneqq ||f_{n+1}^{-1}||^{1/2} \cdot ||f_n||,\\& \alpha_E \coloneqq (4 + 2||d_n||)||g_{n+1}||\cdot||g_n^{-1}||,
\\ &\alpha_1 \coloneqq  (4 + 2||d_n||)^{-1/2} . \end{align}
\end{proposition} Define the $n$-th {\itshape Laplacian} of $(C_*,c_*)$ as \begin{equation}\label{LAPDEF} \Delta_n =\Delta_n[C_*] \coloneqq c_n^*c_n + c_{n-1}c_{n-1}^* \colon C_n \to C_n. \end{equation} Each $\Delta_n$ is a self-adjoint morphism. More precisely, each $C_n$ admits an orthogonal decomposition into closed, $\Gamma$-invariant subspaces $C_n = \ker(\Delta_n) \oplus \ker(c_n)^\perp \oplus \ker(c_{n-1}^*)^\perp$, with respect to which $\Delta_n$ decomposes as the direct sum of self-adjoint endomorphisms $\Delta_n = 0 \oplus (c_n^*)^\perp c_n^\perp \oplus c_{n-1}^\perp (c_{n-1}^*)^\perp$. Using this decomposition, it is also straightforward to check that a complex $(C_*,c_*)$ is Fredholm if and only if each Laplacian $\Delta_n[C_*]$ is Fredholm. A complex $(C_*,c_*)$ is said to be {\itshape $\zeta$-regular/of determinant class} if each Laplacian $\Delta_n$ is of $\zeta$-regular/of determinant class. Just like for ordinary morphisms, the Novikov-Shubin Invaraints can be used to check whether a given complex is of determinant class: Namely, if $(C_*,c_*)$ be a Fredholm complex with positive Novikov-Shubin invariants, i.e. if \ $\alpha_n(C_*) \neq 0$ for each $n \in \mathbb N_0$, $(C_*,c_*)$ is of determinant class. 
 \begin{dfn}[$L^2$-torsion for complexes]\label{TorDef} Let $(C_*,c_*)$ be a Fredholm complex of determinant class. Then
\begin{enumerate} \item if $(C_*,c_*)$ is of finite type, its {\itshape $L^2$-torsion} $T^{(2)}(C_*)$ is defined as 
\begin{equation} T_{(2)}(C_*) \coloneqq  \exp \left( \frac{1}{2}\sum_{n=0}^\infty (-1)^{n+1} \cdot n \cdot \ln (\vndet(\Delta_n)) \right). \end{equation} \item If $(C_*,c_*)$ is not of finite type and $\zeta$-regular, its {\itshape $\zeta$-regularized $L^2$-torsion} $T^{(2)}(C_*) \in \reals_{>0}$ is defined as \begin{equation} T_{(2)}(C_*) \coloneqq  \exp \left( \frac{1}{2}\sum_{n=0}^\infty (-1)^{n+1} \cdot n \cdot \ln (\zvndet(\Delta_n)) \right). \end{equation} \end{enumerate} \end{dfn} 

\subsection{The $L^2$-De Rham complex and analytic torsion}\label{SecDeRha}

Throughout, let $(M,g)$ be a smooth, connected, complete, simply-connected $n$-dimensional Riemannian manifold. Furthermore, fix a subgroup $G < \Isom^+(M,g)$, containing a lattice $\Gamma < G$, so that either one of the following two conditions hold:
\begin{enumerate}[label=(\subscript{A}{{\arabic*}})]
\item $M$ is a $G$-homogeneous space (i.e.\ the action of $G$ on $(M,g)$ by isometries is transitive) and there exists some uniform lattice $\Lambda < G$.
\item $G = \Gamma$ is a uniform lattice. 
\end{enumerate}
Given a complex, finite-dimensional representation $\rho \colon G \to \GL(V)$, consider the trivial bundle \begin{equation*} E = E^\rho \coloneqq M \times V \downarrow M \end{equation*} over $M$, with trivial flat connection. It comes equipped with a flat $G$-action given by \begin{equation*} \gamma . (x,v) \coloneqq (\gamma.x,\rho(\gamma)v). \end{equation*} From now on, we assume that we have found and fixed some $G$-equivariant Hermitian metric $h$ on $E$ and call the resulting Hermitian bundle $E \downarrow M$ {\bfseries $G$-equivariant}. The existence of $h$ is guaranteed when we assume $(A_2)$, or when we assume $(A_1)$ and the point stabilizer $G_x < G$ (for $x \in M$ arbitrary) is a compact subgroup (i.e.\ when $G = \Isom^+(M,g)$ and $(M,g)$ is a symmetric space). Unless explicitly stated otherwise, the examples yet to appear in this paper will all satisfy one of the previously mentioned properties $(A_1)$ and $(A_2)$, and all admit some $G$-equivariant Hermitian metric $h$.
\begin{dfn}\label{HOMDEF} The metric bundle $E \downarrow M$ is called {\itshape locally homogeneous} if for any two $x,y \in M$, such that either $x,y \in M \setminus \partial M$ or $x,y \in \partial M$, there exists open neighborhoods $U \ni x, V \ni y$ and a flat bundle isometry $F \colon E|_{U} \downarrow U \to E|_{V} \downarrow V$ satisfying $F(x) = y$. It is called (globally) {\itshape homogeneous} if one can choose $U = V = M$ in above construction. \end{dfn}
Observe that, if $(M,g)$ satisfies $(A_1)$, the choice of our metric $h$ ensures that the $G$-action on $E \downarrow M$, which is transitive on the base points, is an action by flat bundle isometries. Therefore, we easily deduce
\begin{corollary} If $(M,g)$ satisfies $(A_1)$, then the bundle $E \downarrow M$ is homogeneous. \end{corollary}
 Define $\Omega(M) \coloneqq \bigoplus_{k=0}^n \Omega^k(M)$ to be the ordinary De Rham complex and let \begin{equation*} \Omega(M,E) \coloneqq \Omega(M) \otimes_{C^\infty(M,\ceals)} \Gamma(M,E) \end{equation*} be the De Rham complex of $E$-valued differential forms, with natural grading $\Omega^*(M,E)$ inherited from the grading on $\Omega(M)$ and with differential \begin{equation*} d^* \colon \Omega^*(M,E) \to \Omega^{*+1}(M,E) \end{equation*} induced by the trivial flat connection on $E$ in the usual fashion. We obtain natural $\ceals$-linear actions of $G$ on both $\Omega^*(M)$ and the space of sections $\Gamma(M,E)$, which together induce a natural $\ceals$-linear action of $G$ on $\Omega^*(M,E)$. Moreover, for each $0 \leq p \leq n$, the pair of metrics $g$ and $h$ together give rise to an inner product $\langle \; , \; \rangle_x$ on each fiber vector space $(\Lambda^p T^*M \otimes_{\ceals} E)_x$, and therefore also to an inner product on the subspace $\Omega_c^*(M,E) \subseteq \Omega^*(M,E)$ of compactly supported forms, defined via \begin{equation}\label{IP1} \langle \omega,\sigma \rangle \coloneqq \int_{M} \langle \omega(x),\sigma(x) \rangle_x \: dx.\end{equation} Here, as everywhere else, $dx \in \Omega^n(M)$ denotes the volume form constructed from the Riemannian metric $g$. Evidently, the $G$-action on $\Omega^*(M,E)$ leaves the subspace $\Omega_c^*(M,E)$ invariant, and $G$-equivariance of the two metrics $g$ and $h$ guarantees that this restricted action is by isometries. We define \begin{equation} \Omega_{(2)}^*(M,E) \coloneqq L^2( \Omega_c^*(M,E)) \end{equation} to be the Hilbert space obtained via $L^2$-completion of $\Omega_c^*(M,E)$ and observe that the $G$-action on $\Omega_c^*(M,E)$ extends to a linear isometric $G$-action on $\Omega_{(2)}^*(M,E)$. Moreover, the differential $d^*$ admits a $G$-equivariant {\itshape minimal closure} \cite[Lemma $2.1.3$]{Ich} when regarded as an unbounded, densely defined operator between $\Omega_{(2)}^*(M,E)$ to $\Omega_{(2)}^{*+1}(M,E)$. Since there is no danger of ambiguity, this closure will from now on be denoted by $d^*$.  \\ Restricting the action of $G$ on $\Omega_{(2)}^*(M,E)$ onto the lattice $\Gamma$, it follows that $\Omega_{(2)}^*(M,E)$, along with the $\Gamma$-equivariant closed, unbounded differential $d^*$ has the structure of a Hilbert $\vnN(\Gamma)$-cochain complex. In fact, for any $\Gamma$-fundamental domain $\mathcal F \subseteq M$ with $E|_{\mathcal F} \downarrow \mathcal F$ denoting the restriction of $E$ over $\mathcal F$, $\Omega_{(2)}^*(M,E)$ is $\Gamma$-equivariantly and isometrically isomorphic to the Hilbert space tensor product $L^2(\Gamma) \otimes  \Omega_{(2)}^*(\mathcal F,E|_{\mathcal F})$ (with $\Gamma$-action on $L^2(\Gamma) \otimes  \Omega_{(2)}^*(\mathcal F,E|_{\mathcal F})$ given by the natural action on the left factor). \\
As usual, the Riemannian metric $g$ gives rise to a direct sum decomposition of the restricted cotangent bundle $T^*M|_{\partial M} = T^*\partial M \oplus N^*\partial M$, where $N^*\partial M$ denotes the $1$-dimensional real {\itshape conormal} bundle over $M$. Consequently, we obtain for $0 \leq p \leq n$ a direct sum decomposition \[ (\Lambda^p T^*M \otimes E)|_{\partial M} = (\Lambda^{p} T^*\partial M \otimes E) \oplus (\Lambda^{p-1} T^*\partial M \otimes N^*\partial M \otimes E). \] The projection $\vec{t} \colon  (\Lambda^p T^*M \otimes E)|_{\partial M} \to \Lambda^{p} T^*\partial M \otimes E$  onto the first summand is called the {\itshape tangential boundary projection}, while the projection $\vec{n} \colon (\Lambda^p T^*\partial M \otimes E)|_{\partial M} \to \Lambda^{p-1} T^*\partial M \otimes N^*\partial M \otimes E $ onto the second summand is called the {\itshape normal boundary projection}. The latter projection allows us to define \begin{align} & \delta^* \colon \dom(\delta^*) \to \dom(\delta^{*-1}), \\ & \dom(\delta^*) \coloneqq \{ \omega \in \Omega^*(M,E) \colon \vec{n} \omega = 0 \} \end{align} as the formal adjoint of the differential operator $d$, constructed with respect to the inner product \ref{IP1} and with {\itshape absolute boundary conditions}. Finally, we define the {\itshape Hodge-Laplacian} with absolute boundary conditions as \begin{align} &\Delta_* = \Delta_*[E] \coloneqq \delta^{*+1}d^* + d^{*-1}\delta^* \colon \dom(\Delta_*) \to \dom(\Delta_*), \\ & \dom(\Delta_*) \coloneqq \{ \omega \in \Omega^*(M,E) \colon \vec{n}\omega = \vec{n}d^*\omega = 0  \}. \end{align} When regarded as unbounded, densely defined operators over $\Omega^*_{(2)}(M,E)$, all three maps $\delta^*, d^*$ and $\Delta_*$ admit $G$-equivariant minimal closures \cite[Lemma $2.1.3$]{Ich}. {\itshape Henceforth, the symbols $\delta^*,d^*$ and $\Delta_*$ are used to describe the minimal closures of the corresponding differential operators}. Since each $\Omega^*_{(2)}(M,E)$ is a Hilbert $\vnN(\Gamma)$-module, it follows that \begin{equation} (\Omega_{(2)}^*(M,E),d^*) \colon \Omega_{(2)}^0(M,E) \xrightarrow{d^0} \Omega_{(2)}^1(M,E) \xrightarrow{d^1} \dots \dots \Omega_{(2)}^{n-1}(M,E) \xrightarrow{d^{n-1}} \Omega_{(2)}^n(M,E) \end{equation} is a Hilbert $\vnN(\Gamma)$-cochain complex. Moreover, $\delta^*$ is the Hilbert space adjoint of $d^*$, and vice versa \cite[Proposition $3.4.6$]{Ich}. Finally, $\Delta_* $ equals the $L^2$-Laplacian of the complex $(\Omega_{(2)}^*(M,E),d^*)$ which was defined in \ref{LAPDEF}, and is thus in particular self-adjoint \cite[Theorem $3.4.1$]{Ich}. \\ 
Given a bounded Borel function $\phi$ over $\mathbb R^+$, let $\phi(\Delta_p) \colon \Omega^k_{(2)}(M,E) \to \Omega^k_{(2)}(M,E)$ be the $1$-parameter family of bounded morphisms of Hilbert $\vnN(\Gamma)$-modules constructed via the spectral theorem applied the the self-adjoint $\Delta_p$. The following is well-known to be true:
\begin{enumerate}
\item \cite[Proposition $3.4.2$]{Ich} $\phi(\Delta_p)$ admits a smooth {\itshape kernel} $\phi(\Delta_p)(x,y)$. This is a smooth section of the bundle \\ $\hom\left(\pi_2^*(\Lambda^p T^*M \otimes E),\pi_1^*(\Lambda^p T^*M \otimes E) \right) $ over the product $M \times M$ with (and uniquely determined by) the property that for each $\sigma \in \Omega_c^k(M,E)$, one has for each $x \in M$ \begin{equation} \phi(\Delta_p)\sigma(x) = \int_{M} \phi(\Delta_p)(x,y)\sigma(y) dy. \end{equation}.
\item \cite[Propisition $2.2.1$]{Ich} One has \begin{equation}\label{equivariance} \tr(\phi(\Delta_p)(x,x)) = \tr(\phi(\Delta_p)(\gamma.x,\gamma.x)) \end{equation} for each $\gamma \in G$ and each $x \in M$. Moreover, 
\item  \cite[Proposition $4.2.2$]{Ich} if $\mathcal F \subseteq M$ is a $\Gamma$-fundamental domain, one has \begin{equation} \vntr(\phi(\Delta_p)) = \int_{\mathcal F} \tr(\phi(\Delta_p)(x,x)) dx < \infty. \end{equation}
\end{enumerate}
In particular, it follows that each member of the family $(e^{-t\Delta_p})_{t > 0}$ is of $\Gamma$-trace class, which is why $(\Omega_{(2)}^*(M,E),d^*)$ is a Fredholm complex. In fact, the function $\vntr(e^{-t\Delta_p})$ has an asymptotic expansion near $0$, which can be made (more or less) explicit:
\begin{theorem}[Asymptotic expansion of the twisted heat kernel]\cite[Theorem $4.3.2$]{Ich} \label{asymptbdgeom}
Under the assumptions from above, it holds that for each $0 \leq p \leq n$ and each $0 \leq i \leq n$, there exists forms $\alpha_i[E] \in \Omega^n(M)$ and $\beta_i[E|_{\partial M}] \in \Omega^{n-1}(\partial M)$, such that for $t \to 0$, we have the asymptotic expansion
\begin{align} & \vntr(e^{-t\Delta_p }) =  \sum_{i = 0}^n t^{-(n-i)/2} \left(\int_{\mathcal F} \alpha_i[E] + \int_{\partial \mathcal F}  \beta_i[E|_{\partial M}]  \right) + \mathcal O(t^{1/2}), \end{align} where
\begin{itemize} \item $\mathcal F$ is a fundamental domain for the $\Gamma$-action on $M$, and \item $\partial \mathcal F$ is a fundamental domain for the induced $\Gamma$-action on $\partial M$. \end{itemize} Furthermore, both $\alpha_i[E]$ and $\beta_i[E|_{\partial M}]$ are {\bfseries invariant under local bundle isometries} in the way described as follows: Suppose that $(M',g')$ is another $n$-dimensional complete, oriented Riemannian manifold and $E' \downarrow M'$ a flat Hermitian bundle over $M'$. Then, if  $U \subseteq M$ and $V \subseteq M$ are open subsets, such that there exists a flat bundle isometry $E_U \downarrow U \bundf E'_V \downarrow V$, we have for all $i \in \mathbb N$ \begin{align*}& \alpha_i[E] \equiv \alpha_i[E'] \circ F \hspace{2.1cm} \text{on} \; U, \\& \beta_i[E|_{\partial M}] \equiv \beta_i[E'|_{\partial M'}] \circ F \hspace{1cm} \text{on} \; U \cap \partial M.\end{align*}
\end{theorem}
The fact that the differential forms $\alpha_i[E]$ and $\beta_i[E|_{\partial M}]$ are invariant under local bundle isometries has the following crucial consequence:
\begin{corollary}\label{IMPORTANTE} In the situation of Theorem \ref{asymptbdgeom}, suppose further that the bundle $E \downarrow M$ under inspection is locally homogeneous. Then it follows that for each $0 \leq i \leq n$, there exists constants $a_i,b_i \in \mathbb R_{\geq 0}$, such that \begin{align} & \alpha_i[E] = a_i \cdot d\Vol_g, \label{identa} \\& \beta_i[E] = b_i \cdot d\Vol_{g|_{\partial M}}.  \label{identb} \end{align} The constant $a_i$ depends only on the {\itshape local} isometry type of the restricted bundle $E_{\mathring{M}} \downarrow \mathring{M}$, while the constant $b_i$ depends only on the local isometry type of the restricted bundle $E_{U} \downarrow U$, where $U$ is an arbitrary neighborhood of $\partial M$. \end{corollary}
 Furthermore, we can conclude:
\begin{corollary}\label{zetaregi} If the complex $\Omega_{(2)}^*(M,E)$ is $L^2$-acyclic, then it is also $\zeta$-regular. Namely, for any fixed $0 \leq p \leq n$ the integral expression \begin{align} \zeta_p(s) \coloneqq \zeta_{\Delta_p}(s) = \Gamma(s)^{-1} \int_{0}^{1} t^{s-1} \vntr(e^{-t\Delta_p^\perp[E]}) dt,\end{align} determines a holomorphic function for sufficiently large $s >> 0$, admitting a meromorphic extension onto all of $\ceals$ that is regular at $0$. In fact, it holds that
\begin{align}
& \frac{d}{ds} \zeta_p(s)|_{s=0}  =  \int_{0}^{1} \left( \vntr(e^{-t \Delta_p}) - \sum_{i=0}^n t^{-(n-i)/2} \cdot \kappa_i \right) \: \frac{dt}{t} + \sum_{i = 0}^n c(i,n) \cdot \kappa_i,
\end{align}
where \begin{align}&  \kappa_i \coloneqq \int_{\mathcal F} \alpha_i[E] + \int_{\partial \mathcal F}  \beta_i[E|_{\partial M}], \\& c(i,n) \coloneqq \begin{cases}  \frac{2}{i-n} & i \neq n, \\ -\Gamma'(1) & i = n. \end{cases} \end{align}
\end{corollary}
\begin{proof} Acyclicity of the complex $\Omega_{(2)}^*(M,E)$ implies that we have $\Delta_p = \Delta_p^\perp$. The complex Gamma function $\Gamma$ is well-known to satisfy the identities $\Gamma(1) = 1$ and $\Gamma(s+1) = s \cdot \Gamma(s)$ for any $s \in \ceals$.  Therefore, if $h$ is any function holomorphic at $s = 0$, we can compute
\begin{equation} \label{identity11}\frac{d}{ds} \Gamma(s)^{-1} \cdot h(s)|_{s = 0} = h(0).  \end{equation}
Similarly, the function $(\Gamma(s) \cdot s)^{-1}$ is holomorphic near and vanishes at $s = 0$. Also, its derivative at $s = 0$ can be computed as
\begin{equation} \label{identity12} \frac{d}{ds} (\Gamma(s) \cdot s)^{-1}|_{s=0} = -\Gamma'(1). \end{equation} 
With $\kappa_i$ defined as in the statement of the Corollary, we write for $t > 0$ \begin{equation} \vntr(e^{-t\Delta_p }) =\overbrace{ (\vntr(e^{-t\Delta_p }) - \sum_{i=0}^n t^{-(n-i)/2} \kappa_i)}^{\eqqcolon f(t)} + \sum_{i=0}^n t^{-(n-i)/2} \kappa_i. \end{equation} By Theorem \ref{asymptbdgeom}, the function $f(t)$ lies in $\mathcal O(t^{1/2})$ as $t \to 0$, which is why the $t$-integral  
\\ $H(s) \coloneqq \int_{0}^1 t^{s-1} \cdot f(t) dt$ is holomorphic near $s = 0$. Using \ref{identity11} and \ref{identity12}, we finally compute 
\begin{align*}& \frac{d}{ds} \zeta_p(s)|_{s=0} = \frac{d}{ds}\left(\Gamma(s)^{-1} \cdot H(s)\right)|_{s =0} + \sum_{i=0}^n \kappa_i \cdot \frac{d}{ds}(\Gamma(s)^{-1} \cdot \frac{1}{s - \frac{n - i}{2}})|_{s = 0} \\ & = H(0) + \sum_{i=0}^n \kappa_i \cdot c(i,n) = \int_{0}^{1} \left(\vntr(e^{-t\Delta_p }) - \sum_{i=0}^n t^{-(n-i)/2} \cdot \kappa_i \right) \frac{dt}{t} + \sum_{i=0}^n \kappa_i \cdot c(i,n). 
\end{align*} as claimed. 
\end{proof}

\begin{remark} Observe that, if the metric bundle $E \downarrow M$ is locally homogeneous, we get that \begin{equation} \kappa_i = \Vol_g(\mathcal F) \cdot a_i + \Vol_{g|_{\partial M}}(\partial \mathcal F) \cdot b_i.\end{equation} for each $0 \leq i \leq n$, with $a_i$ and $b_i$ the constants from Corollary \ref{IMPORTANTE}.
 \end{remark}

\begin{dfn} Assume that the Hilbert $\vnN(\Gamma)$-cochain complex $\Omega_{(2)}(M,E)$ is of determinant class. Then the {\itshape analytic $L^2$-torsion} $T_{(2)}^{An}(\Gamma \backslash M,\rho)$  is defined as the regularized $L^2$-torsion of the complex $\Omega_{(2)}(M,E)$, that is \begin{equation} T_{(2)}^{An}(\Gamma \backslash M, \rho) = T^{(2)}(\Omega_{(2)}^*(M,E)) \in \reals_{>0}. \end{equation} \end{dfn}

Although supressed in the notation, notice that $T_{(2)}^{An}(\Gamma \backslash M, \rho)$ {\itshape a priori} still depends on the choice of $g$ and $h$. At times, when highlighting these dependencies becomes relevant, we will therefore denote $T_{(2)}^{An}(\Gamma \backslash M,\rho)$ by $T_{(2)}^{An}(\Gamma M,\rho,g,h)$ instead. 
A substantial amount of this paper will be devoted to showing that the bundles under consideration have de Rham complexes which are of determinant class, thus admit an analytic torsion.

\subsection{The Principle of not feeling the boundary}\label{1.3}

A key technique that we will involve later on lies in the heat kernel comparison between a given manifold $N$ and a closed submanifold $M \subseteq N$ with boundary. Very roughly stated, the next result shows that, within all regions of $M$ that are sufficiently far away $\partial M$, the heat kernels of $M$ and $N$ are relatively similar. While not demonstrated here, this so-called ``Principle of not feeling the boundary'' has also been used to prove the existence of the asymptotic expansions of Theorem \ref{asymptbdgeom}.

\begin{theorem}[Principle of not feeling the boundary]\label{heat}
With the bundle $E \downarrow M$ defined as in the previous section, let $N \subseteq M$ be a (topologically) closed, codimesion $0$ submanifold and let $E|_{N} {\downarrow} N$ be the flat bundle over $N$, obtained by restriction of $E$ to $N$. For $p \geq 0$, let $\Delta_p[E] $ and $\Delta_p[E|_N] $ be the corresponding Bochner-Laplace operators on twisted $p$-forms on the bundle $E$, respectively $E|_N$. For $t > 0$ and $k \in \mathbb N_0$, consider the two bounded operators \begin{align*}& \Delta_p^k[E] e^{-t\Delta_p[E]} \colon \Omega_{(2)}^*(M,E) \to \Omega_{(2)}^*(M,E) \\& \Delta_p^k[E|_{N}] e^{-t\Delta_p[E|_{N}]} \colon \Omega_{(2)}^*(N,E|_{N}) \to \Omega_{(2)}^*(N,E|_{N}), \end{align*} and, for each $x,y \in N$, denote by \begin{align}\Delta_p^k[E]  e^{-t\Delta_p[E] }(x,y): E_x \to E_y, \\  \Delta_p^k[E|_N]  e^{-t\Delta_p[E|_N]}(x,y): E_x \to E_y \end{align} their respective smooth heat kernels. \\
Then the following two results hold true: \begin{enumerate} 
\item  There exists a constant $\kappa > 0$ depending only on the dimension of $M$, and, for each $k \in \mathbb N$ and any $D > 0$, a constant $C_k(D) > 0$, depending only on $D$ and the metric tensors of $g$ and $h$ (but {\bfseries not} on $N$), such that for any pair $x_0,y_0 \in N$ with $d_N(x_0) \coloneqq d(x_0, M \setminus N) \geq D$ and $d_N(y_0) \geq D$, we have the inequality \begin{equation} || \Delta_p[E] ^k e^{-t\Delta[E]}(x_0,y_0) - \Delta_p[E|_N] ^k e^{-t\Delta[E|_N]}(x_0,y_0) || \leq C_k(D) e^{-\frac{d_N(x_0) + d_N(y_0) + 2d(x_0,y_0)}{\kappa t}}. \end{equation}
\item For any $t_0 > 0$, there exists a constant $c(t_0)$, such that for all $t \geq t_0$, we have 
\begin{align} ||e^{-t\Delta_p[E] }(x,y)|| \leq c(t_0), \\ ||e^{-t\Delta_p[E|_N] }(x,y)|| \leq c(t_0). \end{align} \end{enumerate}
 \end{theorem}
 
\begin{proof} For the case that $E$ is the trivial bundle, see \cite[Theorem $2.26$]{Lueck:hyp}. The general case is proven similarly, see \cite[Theorem $3.5.6$]{Ich}. \end{proof}

\subsection{The combinatorial $L^2$-complex and topological $L^2$-torsion}\label{CombTopSec}
Let $\Gamma$ be a countable group acting freely and properly discontinuously on a simply-connected space $M$, so that the quotient $\Gamma \backslash M$ has the homotopy type of a finite CW-complex. It follows that there exists some finite $\Gamma$-CW complex $X$ that has $\Gamma$-homotopy type of the $\Gamma$-space $M$ (so that the quotient spaces $\Gamma \backslash M$ and $\Gamma \backslash X$ are homotopy equivalent). Pick one such $X$. Since $X$ is a CW-complex, we may consider the associated cellular cochain complex $C^*(X,\mathbb C) \coloneqq \bigoplus_{p=0}^n C^p(X,\mathbb C)$ with complex coefficients. \\ The cellular $\Gamma$-action on $X$ endows $C^*(X,\mathbb C)$ with
a structure of a complex of free and finitely generated left-$\mathbb C[\Gamma]$-modules, with $\ceals[\Gamma]$-equivariant differential $\delta^* \colon C^*(X,\mathbb C) \to C^{*+1}(X,\mathbb C)$ and $\mathbb C[\Gamma]$-bases given by explicit representatives for each cell orbit $\pm \Gamma.e$, where $e$ ranges over all the cocells of $X$. For each $p \in \mathbb N_0$, we declare $m_p \in \mathbb N_0$ to be the number of cell orbits of dimension $p$. \\ Next, use the representation $\rho \colon G \to \GL(V)$ to form the tensor product of vector spaces \begin{equation} C^*(X,\rho) \coloneqq C^*(X,\mathbb C) \otimes_{\mathbb C} V. \end{equation} Each space $C^p(X,\rho)$ again also comes equipped with the structure of a (left) $\mathbb C[\Gamma]$-module, induced by the (left) diagonal $\Gamma$-action given by $\gamma.(x \otimes v) \coloneqq \gamma.x \otimes \rho(\gamma)v$ on elementary tensors.
Equipped with the $\ceals[\Gamma]$-equivariant differential \begin{equation*} \delta^*_\rho \coloneqq \delta^* \otimes_{\ceals} \unit_{V} \colon C^*(X,\rho) \to C^{*+1}(X,\rho), \end{equation*} it is still a {\itshape complex} of {\itshape free, finitely generated} left $\mathbb C[\Gamma]$-modules. Namely, let $E \subseteq C^*(X,\ceals)$ be a collection of cocells, precisely one for each cell orbit and each equipped with some orientation. Additionally fixing a $\ceals$-base $B \subseteq V$, one easily checks that the resulting {\itshape finite} set  $E \otimes B = \{ e \otimes b \colon e \in E \wedge b \in B \}$ is a $\ceals[\Gamma]$-base of $C^*(X,\rho)$.\\ In particular, its $\Gamma$-orbit $\Gamma.\left(E \otimes B\right)$ is a $\ceals$-base for $C^*(X,\rho)$  (which is infinite whenever $\Gamma$ is infinite). Equipping $C^*(X,\rho)$ with the unique inner product, with respect to which the set $\Gamma. \left(E \otimes B\right)$ is orthonormal, and forming the corresponding $L^2$-completion, we obtain a cochain complex of Hilbert spaces $C^*_{(2)}(X,\rho)$. It is easily checked that $C^*_{(2)}(X,\rho)$ has the structure of a Hilbert $\vnN(\Gamma)$-cochain complex {\itshape of finite type}. More precisely, the differential $d^*_\rho$ extends to a bounded, $\ceals[\Gamma]$-equivariant operator from $C^*_{(2)}(X,\rho)$ to $C^{*+1}_{(2)}(X,\rho)$ (also denoted by $d^*_\rho$ throughout) and, for each $p \in \mathbb N_0$, there is an isometric isomorphism of Hilbert $\vnN(\Gamma)$-modules \begin{equation} C^p_{(2)}(X,\rho) \cong L^2(\Gamma)^{m_p \cdot \dim(\rho)}. \end{equation} 
The last statement implies that each $C^p_{(2)}(X,\rho)$ is a free, finitely-generated Hilbert $\vnN(\Gamma)$-module with finite {\itshape von Neumann Dimension} 
\begin{equation} \vndim(C^p_{(2)}(X,\rho)) = m_p \cdot \dim(\rho).  \end{equation} We remark that the {\itshape inner product structure} on $C^*_{(2)}(X,\rho)$ clearly depends on the choice of the set $E \otimes B$, while the isomorphism class of the complex $C^*_{(2)}(X,\rho)$ (as a complex of Hilbert $\vnN(\Gamma)$-modules) does {\itshape not}. As such, both $\vndim(C^p_{(2)}(X,\rho))$, as well as any of the spaces and quantities yet to be defined {\itshape apart from the $L^2$-torsion} are independent of the explicit choice of the set $E \otimes B$. 
%The twisted {\itshape $L^2$-cohomology space} $H^p_{(2)}(X,\rho)$ of the complex $C^*_{(2)}(X,\rho)$ at level $p$ is defined as \begin{equation} H^p_{(2)}(X,\rho) = \ker(d^p_{\rho}) / \overline{\im(d^{p-1}_{\rho})}. \end{equation} Via our choice of inner product on $C^p_{(2)}(X,\rho)$, $H^p_{(2)}(X,\rho)$ can be identified with the closed, $\Gamma$-equivariant subspace $\ker(d^p_{\rho}) \cap \im(d^{p-1}_{\rho})^\perp$, implying that $H^p_{(2)}(X,\rho)$ is also a finitely-generated Hilbert $\vnN(\Gamma)$-module. This leads to the next definition: 
\begin{dfn}\label{TopBetDef} For $p \in \mathbb N_0$, the $p$-th {\itshape topological $L^2$-Betti number} $b_{(2),p}^{Top}(\Gamma \backslash M,\rho)$ of the pair $(\Gamma \backslash M,\rho)$ is defined as 
\begin{equation} b_{(2),p}^{Top}(\Gamma \backslash M,\rho) \coloneqq b_p^{(2)}(C^*_{(2)}(X,\rho)). \end{equation} Similarly, the $p$-th {\itshape topological Novikov-Shubin Invariant} $\alpha^{Top}_p(\Gamma \backslash M,\rho)$ of the pair $(\Gamma \backslash M,\rho)$ is defined as \begin{equation} \alpha_p^{Top}(\Gamma \backslash M,\rho) \coloneqq \alpha_p(C^*_{(2)}(X,\rho)). \end{equation} \end{dfn}
As the notation indicates, the numbers $b_{(2),p}^{Top}(\Gamma \backslash M,\rho), \alpha_p^{Top}(\Gamma \backslash M,\rho)$ are independent of the specific choice of $\Gamma$-CW complex $X$ homotopy equivalent to $M$. That is because if $X'$ is another finite $\Gamma$-CW complex that is $\Gamma$-homotopy equivalent to $M$, the induced Hilbert $\vnN(\Gamma)$-cochain complex $C_{(2)}^*(X',\rho)$ is chain homotopy equivalent to $C_{(2)}^*(X,\rho)$. Now we apply Theorem \ref{DILEQ} to conclude that \begin{align*} & b_p^{(2)}(C^*_{(2)}(X,\rho)) = b_p^{(2)}(C^*_{(2)}(X',\rho)), \\ & \alpha_p(C^*_{(2)}(X,\rho)) = \alpha_p(C^*_{(2)}(X',\rho)), \end{align*} as claimed. \\
For the next definition, assume additionally the following:
\begin{enumerate}[label=(\subscript{T}{{\arabic*}})] 
\item The representation $\rho \colon \Gamma \to \GL(V)$ is {\itshape unimodular}, i.e.\ satisfies $|\det \circ \rho| \equiv 1$. 
\item The pair $(M,\rho)$ is {\itshape topologically $L^2$-acyclic}, i.e.\ all twisted $L^2$-Betti numbers $b_{(2),p}^{Top}(M,\rho)$ vanish.
\item The {\itshape Whitehead group} $\Wh(\Gamma)$ (See \cite[Definition 5.3.1]{Ich}) of $\Gamma$ vanishes.
\item $\alpha^{Top}_p(M,\rho) > 0$ for each $p \in \mathbb N_0$.
\end{enumerate}
In this case, we can define the {\itshape twisted topological} $L^2$-torsion of the pair $(\Gamma \backslash M,\rho)$ as \begin{equation} T^{Top}_{(2)}(\Gamma \backslash M,\rho) \coloneqq T_ {(2)}(C_{(2)}^*(X,\rho)) \in \reals_{>0}, \end{equation} where $T_ {(2)}(C_{(2)}^*(X,\rho))$ is the $L^2$-torsion for finite type Hilbert $\vnN(\Gamma)$ cochain complexes introduced in Definition \ref{TorDef} (see \cite[Chapter 5.3]{Ich} for a detailed explanation). Just like for $b^{Top}_{(2),p}$ and $\alpha^{Top}$, under the assumptions $(1)-(4)$, one can show that $T^{Top}_{(2)}(M,\rho)$ is independent of both the choice of base on $C_{(2)}^*(X,\rho)$, as well as of the model space $X$ itself.

\section{Some known relations between different $L^2$-invariants}\label{RelL2}
 
\subsection{The induction principle} 
Let $\Gamma_0 < \Gamma$ be a subgroup and denote by $\rho_0$ the restriction of $\rho$ to $\Gamma_0$. Choose $X$ to be a finite $\Gamma_0$-CW complex and let $C^*(X,\rho_0)$ be the associated Hilbert $\vnN(\Gamma_0)$-cochain complex of finite type, whose construction was laid out in the previous section. We now explain how $C^*(X,\rho_0)$ gives rise to a free, finite Hilbert $\vnN(\Gamma)$-cochain complex via the {\itshape principle of induction}: Since $\Gamma_0 < \Gamma$, we can naturally regard the group ring $\mathbb C[\Gamma]$ as a {\itshape right} $\mathbb C[\Gamma_0]$-module. Hence, we can define the following cochain complex of free, finitely generated left $\mathbb C[\Gamma]$-modules \begin{equation}  C^*(X,\rho_0,\Gamma) \coloneqq \mathbb C[\Gamma] \otimes_{\mathbb C[\Gamma_0]} C^*(X,\rho_0). \end{equation} Similarly as before, the canonical inner product on the group ring $\mathbb C[\Gamma]$ turns $C^*(X,\rho_0,\Gamma)$ into a complex of inner product spaces, whose $L^2$-completion we denote by $C^*_{(2)}(X,\rho_0,\Gamma)$. It is a free, finite Hilbert $\vnN(\Gamma)$-cochain complex. It follows from \cite[Lemma 1.24]{Lueck:Book} that, for each $p \in \mathbb N$, we have \begin{align}&\label{bnoveq1} b_{(2)}^p\left(C_{(2)}^*(X,\rho_0,\Gamma)\right) = b_{(2)}^p\left(C_{(2)}^*(X,\rho_0)\right), \\ &\label{bnoveq2} \alpha_p\left( C^*_{(2)}(X,\rho_0,\Gamma)\right) = \alpha_p\left( C^*_{(2)}(X,\rho_0)\right). \end{align}
Now consider the principal $\Gamma$-bundle \begin{equation} Y \coloneqq \Gamma \times_{\Gamma_0} X. \end{equation} The $CW$-structure on $X$ then extends to a free, finite $\Gamma$-CW structure on $Y$. Just as above, form the twisted $L^2$-cochain complex $C_{(2)}^*(Y, \rho)$. It is proven in \cite[Lemma 1.1, Theorem 6.7,(5)]{Lück:twist} that there is an isomorphism of Hilbert $\mathcal N(\Gamma)$-cochain complexes \begin{equation} C^*_{(2)}(Y, \rho) \cong C^*_{(2)}(X,\rho_0,\Gamma). \end{equation} Using Theorem \ref{DILEQ}, along with Equations \ref{bnoveq1} and \ref{bnoveq2}, we thus obtain for each $p \in \mathbb N$ the equalities \begin{align} \label{bnoveq3} b_{(2)}^p\left(C_{(2)}^*(Y, \rho)\right) = b_{(2)}^p \left(C^*_{(2)}(X,\rho_0)\right), \\ 
\label{bnoveq4} \alpha_p \left(C_{(2)}^*(Y, \rho)\right) = \alpha_p \left(C_{(2)}^*(X,\rho_0)\right). \end{align}
Combined with the results from the previous subsection, we obtain: 
\begin{lemma}\label{indprinc} Let $\Gamma_0 < \Gamma$ be a group acting freely and properly discontinuously on a simply-connected space $M_0$, so that $\Gamma_0 \backslash M_0$ has the homotopy type of a finite CW-complex. Let $\rho \colon \Gamma \to \GL(V)$ be a finite-dimensional representation and denote by $\rho_0$ its restriction to $\Gamma_0$. Then the space $M \coloneqq \Gamma \times_{\Gamma_0} M_0$ is simply-connected, so that $\Gamma \backslash M$ has the homotopy type of a finite CW-complex. Moreover, for each $p \in \mathbb N_0$, we have
\begin{align*}& b_{(2),p}^{Top}(\Gamma \backslash M,\rho) = b_{(2),p}^{Top}(\Gamma_0 \backslash M_0,\rho_0), \\& \alpha_p^{Top}(\Gamma \backslash M,\rho) = \alpha_p^{Top}(\Gamma_0 \backslash M_0,\rho_0). \end{align*} \end{lemma}

\subsection{A de Rham Theorem and an acyclic complex of determinant class}
Let $(M,g)$ be a smooth, simply-connected Riemannian manifold, and let $\Gamma \subset \Isom(M,g)$ be a uniform lattice. A $\Gamma$-CW structure $X$ on $M$ is called {\itshape admissible} if the intersection $X \cap \partial M$ is a $\Gamma$-CW structure for $\partial M$. In all our relevant instances, admissible $\Gamma$-CW structures always exist and can be constructed by lifting appropriate CW-structures from the quotient manifold $\Gamma \backslash M$. 
\begin{theorem}\label{DeRh} \cite[Theorem $6.3.5$]{Ich} Let $(M,g)$ be a simply-connected Riemannian manifold and let $\Gamma \subset \Isom^+(M,g)$ be a uniform lattice. Further, let $\rho: \Gamma \to \GL(V)$ be some finite-dimensional, complex representation and let $E^\rho \coloneqq M \times V \downarrow M$ be the associated flat, $\Gamma$-equivariant bundle over $M$. 
Let $\Omega_{(2)}^*(M,E^\rho,g,h)$ be the Hilbert $\vnN(\Gamma)$-cochain complex (with absolute boundary conditions), constructed with respect to some choice of $\Gamma$-equivariant Hermitian form $h$. Then, for any admissible $\Gamma$-CW structure $X$ on $M$, there is a chain homotopy equivalence of Hilbert $\vnN(\Gamma)$-cochain complexes
\begin{equation} \Omega_{(2)}^*(M,E^{\rho},g,h) \simeq C^*_{(2)}(X,\rho).\end{equation} 
\end{theorem}
Applying Theorem \ref{DILEQ}, we deduce: 
\begin{corollary}\label{DeRhCor} Let $(M,g)$ be a simply-connected Riemannian manifold, let $\Gamma \subset \Isom^+(M,g)$ be a uniform lattice and $\rho \colon \Gamma \to \GL(V)$ a finite-dimensional complex representation. Then 
\begin{align*}& b_{(2),p}^{An}(\Gamma \backslash M,\rho) = b_{(2),p}^{Top}(\Gamma \backslash M,\rho) \\& \alpha_p^{An}(\Gamma \backslash M,\rho) = \alpha_p^{Top}(\Gamma \backslash M, \rho). \end{align*} \end{corollary}
Let $(M,g)$ be a compact, simply-connected Riemannian manifold (possibly with boundary). Given a smooth, positive function $f \colon M \to \mathbb R_{>0}$. and some $k \in \mathbb N$, consider the product manifold $M \times \mathbb R^k$, equipped with the warped product metric \begin{equation*} g' = g + f \cdot dx^2, \end{equation*} where $dx^2$ is the standard Euclidean metric on $\mathbb R^k$. Observe that there is a uniform lattice $\Gamma < \Isom(M \times \mathbb R^k, g')$ isomorphic to $\mathbb Z^k$, whose action is on $M \times \mathbb R^k$ is trivial on the first, and given by the standard coordinate translations on the second factor. Consequently, the quotient space $\Gamma \backslash (M \times \mathbb R^k)$ is homeomorphic to $M \times T^k$, with $T^k$ denoting the $k$-dimensional torus. For the flat bundle $E^\rho \downarrow M \times \mathbb R^k$ associated with a representation $\rho \colon \Gamma \to \GL(V)$, choose some $\Gamma$-equivariant Hermitian form $h$. Then we may conclude as follows:
\begin{proposition}\label{novpos1} With the notation from above, it holds that for {\bfseries any} finite-dimensional representation $\rho \colon \Gamma \to \GL(V)$, the Hilbert $\vnN(\Gamma)$-cochain complex
$\Omega_{(2)}^*(M \times \mathbb R^k,g',h,E^\rho)$ is acyclic and has positive Novikov-Shubin invariants. More precisely, we have for all $0 \leq p \leq \dim(M) + k$ \begin{align*}
& b^{An}_{(2),p}(M \times T^k,\rho) = b^{Top}_{(2),p}(M \times T^k,\rho) = 0, \\& \alpha^{An}_{p}(M \times T^k,\rho) = \alpha^{Top}_p(M \times T^k,\rho) > 0. \end{align*} \end{proposition}
\begin{proof} The left-hand equalities follow directly from Corollary \ref{DeRhCor}. 
To show the right-hand equalities, the isomorphism $\Gamma \cong \mathbb Z^k$ will be of central importance. Firstly, \cite[Theorem 7.7]{Lück:twist} implies that \begin{equation} b_{(2),p}^{Top}(M \times T^k, \rho)  = \dim_{\mathbb C}(\rho) \cdot b_{(2),p}^{Top}(M \times T^k, \unit), \end{equation} where $\unit: \Gamma \to \ceals^\times$ denotes the trivial representation. From the multiplicativity of ordinary $L^2$-Betti numbers under coverings, see \cite[Example 1.37]{Lueck:Book}, it follows that any compact manifold $X$ that admits non-trivial self-coverings satisfies $b_{(2),p}^{Top}(X,\unit) = 0$ for all $p \geq 0$. Since $M \times T^k$ clearly admits non-trivial self-coverings, we obtain that $b_{(2),p}^{Top}(M \times T^k,\unit) = 0$, from which we conclude that $ b^{An}_{(2),p}(M \times T^k,\rho) = 0$ as desired. 
Secondly, observe that for any $\Gamma$-$CW$-structure $X$ on $M \times \mathbb R^k$, the boundary operators of $C_{(2)}^*(X,\rho)$ are by definition matrices over $\mathbb C[\Gamma] \cong \mathbb C[\mathbb Z^d]$ (acting by right-multiplication). It is shown in \cite[Theorem 1.2]{Lueck:Novi}, that any such matrix has positive Novikov-Shubin invariant. Therefore \begin{equation} \alpha_p^{Top}(M \times T^k,\rho) = \alpha_p(C_{(2)}^*(X,\rho)) > 0 \end{equation} for each $0 \leq p \leq \dim(M) + k$, finishing the proof.
\end{proof}

\subsection{The Cheeger-M\"uller Theorem for uniform lattices}

A subject that has undergone much investigation in the last three decades, which is also a key player of this paper, is the {\itshape $L^2$-torsion anomaly} \begin{equation}\log \left( \frac{T_{(2)}^{An}(M,\rho)}{T^{Top}_{(2)}(M,\rho)} \right) \in \reals \end{equation} constructed from a compact manifold $M$ and a representation $\rho$ of its fundamental group. For completion, we should mention that, although not clear from our definitions, this anomaly can be defined even if the relevant combinatorial and de Rham cochain complexes under inspection are not acyclic and of determinant class. \\
Although the main manifold under inspection is itself not compact, it still admits an exhaustion by compact submanifolds with boundary, for which the author has been able to provide an adequate description of the $L^2$-torsion anomaly in \cite{Ich2}. \\
We won't describe it in full detail, but instead focus on the property that is most relevant for our purposes: Up to a factor given by the dimension of the representation, the $L^2$-torsion anomaly is the same for all representations that are ``admissible'' in the sense that they satisfy the conditions listed in the theorem below. 

\begin{theorem}\label{CheMulComp}\cite[Theorem 1.2]{Ich2} Let $(M,g)$ be a simply-connected Riemannian manifold and let $\Gamma \subset \Isom^+(M,g)$ be a uniform lattice. For $i=1,2$, let $\rho_i: \Gamma \to \GL(V)$ be some finite-dimensional, complex representation and let $E^{\rho_i} \coloneqq M \times V \downarrow M$ be the associated flat, $\Gamma$-equivariant bundle over $M$. 
Let $\Omega_{(2)}^*(M,E^{\rho_i},g,h_i)$ be the Hilbert $\vnN(\Gamma)$-cochain complex (with absolute boundary conditions), constructed with respect to some choice of $\Gamma$-equivariant Hermitian form $h_i$. Suppose that $h_i$ is unimodular as in \cite[Definition 4.2]{Ich2}. If both complexes $\Omega_{(2)}^*(M,E^{\rho_i},g,h_i)$ and $\Omega_{(2)}^*(\partial M,E^{\rho_i}_{\partial M},g,h_i)$ are of determinant class and $\Omega_{(2)}^*(M,E^{\rho_i},g,h_i)$ is $L^2$-acyclic, then 
\begin{equation} \dim(\rho_2) \cdot \log \left( \frac{T_{(2)}^{An}(\Gamma \backslash M,\rho_1,g,h_1)}{T^{Top}_{(2)}(\Gamma \backslash M,\rho_1)} \right) = \dim(\rho_1) \cdot \log \left( \frac{T_{(2)}^{An}(\Gamma \backslash M,\rho_2,g,h_2)}{T^{Top}_{(2)}(\Gamma \backslash M,\rho_2)} \right). \end{equation} \end{theorem}

Finally, although there are a lot of manifolds $(M,g)$ that possess no acyclic representations $\rho$ (for example if $\chi(M) \neq 0$), we should remark that the other critical assumption in the above theorem, the determinant class condition, is conjectured to always be satisfied for all finite-dimensional acyclic representations $\rho$.

\section{The flat, canonical $\rho$-bundle over $\hyp^n$}\label{FLATRHO}

For $n \in \mathbb N$ odd, we set $G \coloneqq SO_0(n,1)$ and let $K \coloneqq SO(n) \subseteq G$. Then $K$ is a maximal compact subgroup of $G$ and we can identify the quotient $G/K$ with the $n$-dimensional hyperbolic space $\hyp^n$. Conversely, we can identify $G$ with $\Isom_0(\hyp^n)$, the identity component of the hyperbolic isometry group. \\ Let $\Gamma \subseteq G$ be a non-uniform lattice. Here, as everywhere else in this paper, lattices are always assumed to be {\itshape torsion-free} (this way, the induced quotient map $\hyp^n \to \Gamma \backslash \hyp^n$ is an honest covering projection). Throughout, $g$ will denote the hyperbolic metric on $\hyp^n$. It is well-known \cite[Chapter 4]{Petronio:hyp} that, associated with $\Gamma$, we then find a totally ordered set \begin{equation}\label{thick} \{M_R \subseteq \hyp^n: R \in [0,\infty) \} \end{equation} of complete $\Gamma$-invariant submanifolds of $\hyp^n$ (with $M_R \subset M_{R'}$ if $R < R'$ ), such that, additionally, \begin{enumerate}
\item $\hyp^n = \bigcup_{R > 0} M_R$,
\item $\Gamma$ acts cocompactly on each $M_R$,
\item the complete submanifold \begin{equation}\label{thin} C_R \coloneqq \clos(\hyp^n \setminus M_R) \end{equation} is also $\Gamma$-invariant. Moreover, there exists an integer $k \in \mathbb N$ and, for each $1 \leq j \leq k$, complete, connected submanifolds $C_0^j$ of $C_0$ with $C_R^j \coloneqq C_R \cap C_0^j$ complete, connected submanifolds of $C_R$ for each $R \geq 0$, such that the following holds: \begin{enumerate}
\item $C_0^j \cong [0,\infty) \times \reals^{n-1}$ under a diffeomorphism that identifies $C_R^j$ with $[R,\infty) \times \reals^{n-1}$. Furthermore, under the aforementioned identification, the hyperbolic metric restricted to $C_0^j$ is of the form \begin{equation}\label{warpprod} dt^2 + e^{-2t}dx^2, \end{equation} where $dt^2$ is the Euclidean metric on $[0,\infty)$ and $dx^2$ the Euclidean metric on $\reals^{n-1}$.
\item For each $R \geq 0$, we have an equality of stabilizer subgroups $\Gamma_0^j \coloneqq \Gamma_{C^j_0} = \Gamma_{C^j_R} < \Gamma$. The action of $\Gamma_0^j$ on $C_0^j\cong [0,\infty) \times \reals^{n-1}$ is the product of the trivial action  on the first factor $[0,\infty)$ and a cocompact, free, properly discontinuous action by Euclidean isometries on the second factor $\reals^{n-1}$ of $C_0^j$. In particular, $\Gamma_0^j$  is isomorphic to $\mathbb Z^{n-1}$.
\item For each $R \geq 0$, we have an isometric diffeomorphism of principal $\Gamma$-bundles \begin{equation} C_R \cong \coprod_{j=1}^k \Gamma \times_{\Gamma_0^j} C_R^j. \end{equation}
\end{enumerate}
\end{enumerate}
\begin{remark} In fact, the above decomposition of $\hyp^n$ into $\Gamma$-invariant parts still holds true if $\Gamma$ is uniform (i.e.\ $\Gamma$ acts cocompactly on $\hyp^n$) for trivial reasons. Namely, in this instance, we can simply define $M_R \coloneqq \hyp^n$ for all $R \geq 0$. \end{remark} 
\begin{example} Below left, we have sketched the decomposition of $\hyp^2$ as defined above, along with a fundamental domain for the lattice \[ \Gamma \coloneqq \left< \begin{pmatrix} 1 & 0 \\ 2 & 1 \end{pmatrix}, \begin{pmatrix} 1 & 2 \\ 0 & 1\end{pmatrix} \right> < \Isom^+(\hyp^2), \] where the action on $\hyp^2$ is given by M\"obius transformations. In this instance, we have $k = 3$. The colors indicate which horoballs are identified in the quotient space $\Gamma \backslash \hyp^2$, sketched below right, which is homeomorphic to a three-holed sphere. \end{example}
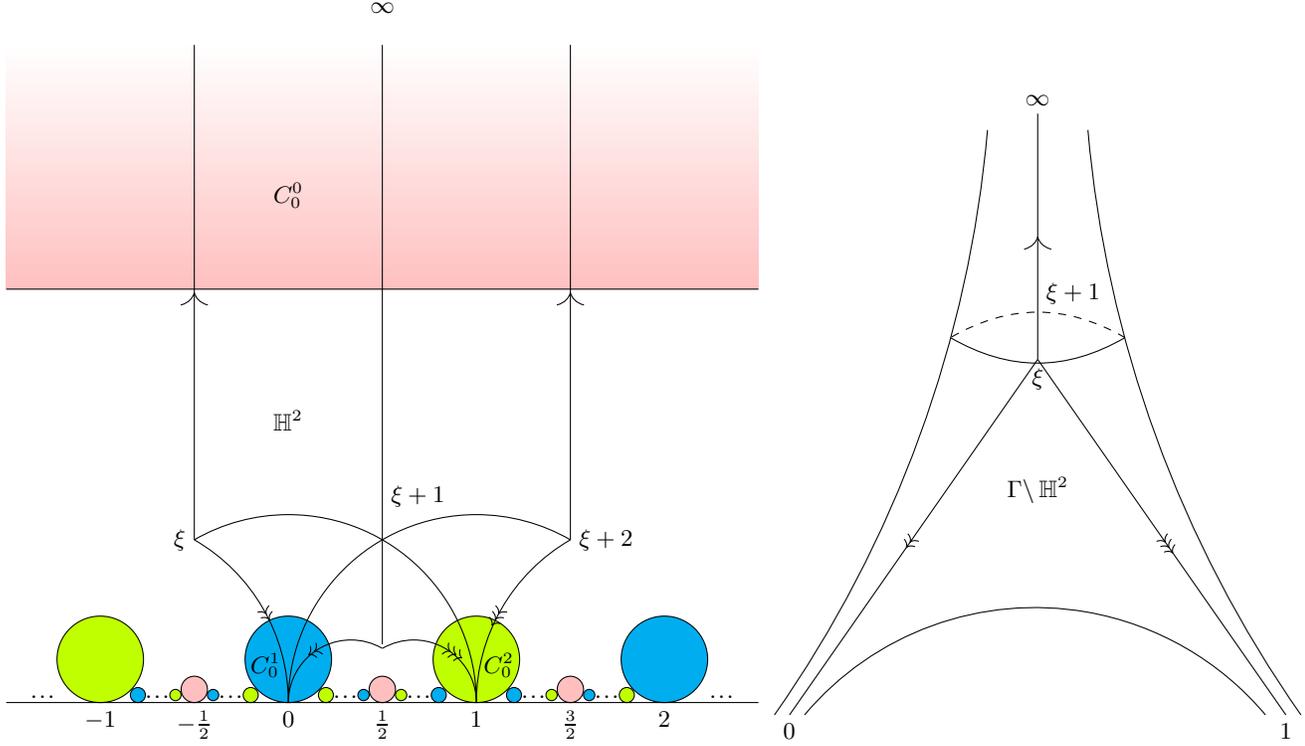
\begin{align*}
\begin{tikzpicture}[scale=2.5]
%Kreise
\filldraw[fill=cyan, draw=black](0,0.23) circle [radius=0.23];
\filldraw[fill=lime, draw=black](0.2,0.04) circle [radius=0.04];
\filldraw[fill=cyan,draw=black](0.4,0.04) circle [radius=0.03];
\filldraw[fill=pink, draw=black](0.5,0.07) circle [radius=0.07];
\filldraw[fill=lime, draw=black](0.6,0.04) circle [radius=0.03];
\filldraw[fill=cyan,draw=black](0.8,0.04) circle [radius=0.04];
%\filldraw[fill=pink, draw=black](0.75,0.05) circle [radius=0.04];
%\filldraw[fill=pink, draw=black](0.23,0.05) circle [radius=0.04];
\filldraw[fill=lime, draw=black](1,0.23) circle [radius=0.23];
\filldraw[fill=cyan,draw=black](1.2,0.04) circle [radius=0.04];
\filldraw[fill=lime, draw=black](1.4,0.04) circle [radius=0.03];
\filldraw[fill=pink, draw=black](1.5,0.07) circle [radius=0.07];
\filldraw[fill=cyan,draw=black](1.6,0.04) circle [radius=0.03];
\filldraw[fill=lime, draw=black](1.8,0.04) circle [radius=0.04];
\filldraw[fill=cyan,draw=black](2,0.23) circle [radius=0.23];
\filldraw[fill=lime, draw=black](-1,0.23) circle [radius=0.23];
\filldraw[fill=cyan,draw=black](-0.8,0.04) circle [radius=0.04];
\filldraw[fill=lime, draw=black](-0.6,0.04) circle [radius=0.03];
\filldraw[fill=pink, draw=black](-0.5,0.07) circle [radius=0.07];
\filldraw[fill=cyan,draw=black](-0.4,0.04) circle [radius=0.03];
\filldraw[fill=lime, draw=black](-0.2,0.04) circle [radius=0.04];
%Knoten
\node[below] (A) at (0,0) {$0$};
\node[below] (B) at (1,0) {$1$};
\node[above] (C) at (0.5,3){$\infty$};
\node[below] (D) at (2,0) {$2$};
\node[below](E) at (0.5,0) {$\frac{1}{2}$};
\node[below](F) at (-0.5,0) {$-\frac{1}{2}$};
\node[below](G) at (-1,0) {$-1$};
\node[below](H) at (1.5,0) {$\frac{3}{2}$};
\node[right](J) at (0.5,1.1) {$\xi +1$};
\node[above](K) at (1.695,-0.02) {$\mydots$};
\node[above](K) at (1.305,-0.02) {$\mydots$};
\node[above](K) at (-1.305,-0.02) {$\mydots$};
\node[above] at (2.305,-0.02) {$\mydots$};
\node[above](K) at (0.695,-0.02) {$\mydots$};
\node[above](K) at (0.305,-0.02) {$\mydots$};
\node[above](K) at (-0.305,-0.02) {$\mydots$};
\node[above](K) at (-0.695,-0.02) {$\mydots$};
\node at (-.12,.2) {$C^1_0$};
\node at (1.12,.2) {$C^2_0$};
\node at (0,1.5) {$\hyp^2$};
\shade[top color=white, bottom color =pink] (-1.5,2.2) rectangle (2.5,3.5);
\node at (0,2.7) {$C^0_0$};
\node at (0.5,3.7) {$\infty$};
\draw[black](-1.5,2.2) --(2.5,2.2);
\draw (-1.5,0) -- (2.5,0);
\draw(0,0)  arc[radius = 1, start angle= 180, end angle= 60] node[right](K) {$\xi+2$};
\draw (1,0) arc[radius =1, start angle =0,end angle =120] node[left](I) {$\xi$};
\draw(0,0)[middlearrow={<<}] arc[radius =1, start angle = 0, end angle =60];
\draw(1,0)[middlearrow={<<<}] arc[radius =1, start angle = 180, end angle = 120];
\draw[middlearrow={<<}](0,0) arc[radius = 0.333, start angle = 180, end angle = 60];
\draw[middlearrow={<<<}](1,0) arc[radius = 0.333, start angle = 0, end angle = 120];
\draw(0.5,0.31) -- (0.5,3.5);
%\draw[dashed] (-0.5,0) -- (-0.5,0.866);
%\draw[dashed] (1.5,0) -- (1.5,0.866);
\draw[middlearrow={>[scale=2.0]}](1.5,0.866) -- (1.5,3.5);
\draw[middlearrow={>[scale=2.0]}](-0.5,0.866) -- (-0.5,3.5);
\end{tikzpicture}
\hspace{.2cm}
\begin{tikzpicture}[scale = 2]
\draw (-2,0) arc[radius = 2, start angle=140, end angle=40]; 
\draw (-2.2,0) arc[radius = 8,start angle=325, end angle = 355];     
\draw (1.3,0) arc[radius = 8,start angle=215, end angle = 185];
\draw[bend left] (0.13,2.51) to  (-1.03,2.51);    
\draw[dashed, bend right] (0.13,2.51) to  (-1.03,2.51);
\draw[middlearrow={>[scale=2.0]}] (-0.45,2.365) -- (-0.45,4) node[above] {$\infty$};
\draw[middlearrow={<<}] (-2.1,0) node[below] {$0$} --(-0.45,2.365);
\draw[middlearrow={<<<}] (1.2,0) node[below] {$1$} -- (-0.45,2.365);
\node[below] at (-0.45,2.365) {$\xi$};
\node at (-0.45,1.5) {$\Gamma \backslash \hyp^2$};
\node[above right] at (-0.45,2.675) {$\xi + 1$};
\end{tikzpicture}
\end{align*}
For each $R \geq 0$ and each $1 \leq j \leq k$, we further define the complete submanifolds \begin{align} \label{torus} T_R \coloneqq C_R \cap M_{R+1}, \\ T_R^j \coloneqq C_R^j \cap T_R.  \end{align} From the above, it follows that each $T_R$ is $\Gamma$-invariant, and that the stabilizer of $T_R^j$ inside $\Gamma$ equals $\Gamma_0^j$.
Moreover, we can identify $T_R^j$ with $[R,R+1] \times \reals^{n-1}$ and the hyperbolic metric correspondingly with $dt^2 + e^{-2t}dx^2$. Finally, it follows that also $T_R$ is a principal $\Gamma$-bundle, isometrically diffeomorphic to  $\coprod_{j=1}^k \Gamma \times_{\Gamma_0^j} T_R^j$.
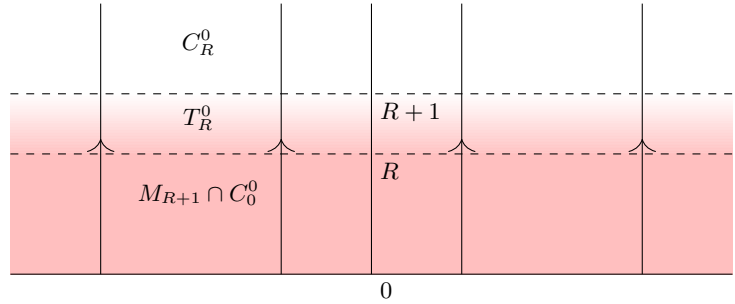
\begin{align*}
\begin{tikzpicture}[scale=0.4]
\fill[pink](-12,0) rectangle (12,4);
%\fill[orange](-3,6) rectangle (3,9);
\shade[bottom color=pink, top color=white] (-12,4) rectangle (12,6);
%\shade[left color=orange, right color=white] (-2.99,9) rectangle (2.99,11);
\draw(0,0) node[below right]{$0$} -- (0,9);
\draw(-12,0) -- (12,0);
\draw[dashed](-12,4) -- (0,4) node[below right]{$R$} -- (12,4);
\draw[dashed](-12,6) -- (0,6) node[below right]{$R+1$} -- (12,6);
\node[above](A) at (-5.75,2) {$M_{R+1} \cap C_0^0$};
\node[above](B) at (-5.75,4.5) {$T_R^0$}; 
\node[above](C) at (-5.75,7) {$C_R^0$};
\draw[middlearrow={>[scale=2.0]}] (3,0) -- (3,9);
\draw[middlearrow={>[scale=2.0]}] (9,0) -- (9,9);
%\draw[middlearrow={>>}] (1,0) -- (1,9);
%\draw[middlearrow={>>}] (-1,0) -- (-1,9);
\draw[middlearrow={>[scale=2.0]}] (-3,0) -- (-3,9);
\draw[middlearrow={>[scale=2.0]}] (-9,0) -- (-9,9);
\end{tikzpicture}
\end{align*}
Consider an irreducible representation $\rho: G \to \GL(V)$ of $G$ on some complex, finite-dimensional vector space $V$. Observe that $\rho$ gives rise to a diagonal action of $G$ on the product $\hyp^n \times V$. Evidently, this determines an action on the flat vector bundle $\hyp^n \times V \downarrow \hyp^n$ by flat bundle isomorphisms, so that the projection map becomes $G$-equivariant (with respect to the $G$-actions on the base space and the total space). By \cite[Lemma 3.1]{Matsushima:Metric}, there exists a distinguished {\itshape Hermitian metric} $h_\rho \in C^\infty(\hyp^n,\GL(V,V^*))$, which satisfies \begin{equation} \langle v,w \rangle_{h_\rho(p)} = \langle \rho(\gamma) \cdot v, \rho(\gamma) \cdot w \rangle_{h_\rho(\gamma.p)}.  \end{equation} The resulting Hermitian bundle $(\hyp^n \times V,h_\rho) \downarrow \hyp^n$ is called the {\bfseries flat, canonical $\rho$-bundle} over $\hyp^n$ and is briefly denoted by $E^\rho \downarrow \hyp^n$. Observe that, with our choice of metric $h_\rho$ on $E^\rho$, the action of $G$ on $E^\rho$, which is transitive on the basepoints, is by flat bundle isometries, so that {\itshape the metric bundle $E^\rho \downarrow \hyp^n$ is $G$-equivariant and homogeneous} (cf.\ Definition \ref{HOMDEF}). \\
For $X \subseteq \hyp^n$ a complete, codimension $0$ hyperbolic submanifold, we let $E^\rho_X \downarrow X$ be the Hermitian {\itshape restriction bundle} of $E^\rho$ over $X$, obtained by pulling back the Hermitian bundle $E^\rho$ through the inclusion $X \hookrightarrow \hyp^n$.\\ % Let $\Omega^{\bullet}(X,E^\rho_X)$ be the de Rham complex of $E^\rho_X$-valued differential forms over $X$ (with pulled-back differential and inner product).
 We set $G_X \coloneqq \{ \gamma \in G: \gamma . X = X \} < G$ to be the subgroup of $G$ leaving $X$ invariant and show the following:
\begin{lemma}\label{CONNECEQ} Let $X,Y \subseteq \hyp^n$ be two connected, codimesion $0$ Riemannian submaifolds of $\hyp^n$ and let $f: X \to Y$ be an orientation-preserving isometry. Then there exists a global isometry $\gamma \in G$, such that \begin{equation*} f = \gamma|_{X}. \end{equation*} In particular, $f$ extends to a flat bundle isometry $F: E_X^\rho \to E_Y^\rho$, where $F$ is given via \begin{equation*} F(x,v) =( \gamma_f . x, \rho(\gamma_f)\cdot v). \end{equation*} {\bfseries Therefore, if $X$ ist connected, the Hermitian bundle $E^\rho_X \downarrow X$ is $G_X$-equivariant. Moreover, if $X$ is locally homogeneous, so is $E^\rho_X \downarrow X$}. \end{lemma}
\begin{proof} Let $\mathring{X}$ be the manifold interior of $X$. Since $X$ has codimension $0$, it suffices to show that $f|_{\mathring{X}} = \gamma|_{\mathring{X}}$ for some global isometry $\gamma \in G$. Thus, let $x \in \mathring{X}$ and let $f^*_x: T_x\hyp^n \to T_{f(x)}\hyp^n$ be the differential of $f$ at $x$. Since the hyperbolic exponential map $\exp_y^{\hyp^n}\colon T_y\hyp^n \to \hyp^n$ is a diffeomorphism for any $y \in \hyp^n$ (which follows, for example, from the {\itshape Cartan-Hadarmard Theorem}), we can define a global isometry $\gamma^{(x)}: \hyp^n \to \hyp^n$ as \begin{equation*}\gamma^{(x)} \coloneqq \exp_{f(x)}^{\hyp^n} \circ f^*_x \circ (\exp_x^{\hyp^n})^{-1}. \end{equation*} To see that $\gamma^{(x)}$ is indeed an isometry, observe first that by definition, $\gamma^{(x)}$ sends all geodesics passing through $x$ isometrically onto geodesics passing through $f(x)$, also preserving the angle between any two of them. Moreover, the {\itshape hyperbolic law of cosines} (see e.g.\ \cite{Reid:Hyp}) readily implies that any hyperbolic triangle $T \subset \hyp^n$ is uniquely determined up to isometry by the length of two of its sides and the interior angle between them. Now let $a,b \in \hyp^n$ be arbitrary points. Without loss of generality, we may assume that $a$ and $b$ don't already lie on a common geodesic passing through $x$. Then there exists a hyperbolic triangle $\Delta(a,b,x)$ with vertices $a,b,x$. By what we've said before, it follows that $\Delta(a,b,x)$ is isometric to the triangle $\Delta(\gamma^{(x)}(a),\gamma^{(x)}(b),\gamma^{(x)}(x))$. In particular, we get $d_g(a,b) = d_g(\gamma^{(x)}(a),\gamma^{(x)}(b))$, finally showing that $\gamma^{(x)}$ is indeed an isometry. \\ Our goal now is to show that for any two points $x,y \in \mathring{X}$, $\gamma^{(x)} = \gamma^{(y)} \equiv f$ in an open subset of $\mathring{X}$ containing $x$ and $y$. Since global isometries on $\hyp^n$ agree everywhere if they agree on a non-empty open subset, we conclude that  $\gamma^{(x)} = \gamma^{(y)} \eqqcolon \gamma$ for all $x,y \in X$, hence the result. \\ To do so, we show in a first step that $f \equiv \gamma^{(x)}$ locally around any $x \in \mathring{X}$. For this, observe that $(\gamma^{(x)})^*_x = f^*_x$ via Gauss' Lemma. Since $X$ has codimension $0$ inside $\hyp^n$, we can choose a small convex ball $B_r(x) \subseteq X$ around $x$ which is mapped isometrically via $f$ onto $B_r(f(x)) \subseteq Y$. For arbitrary $y \in B_r(x)$, we let $\alpha \colon [0,1] \to X$ be the unique geodesic contained in $B_r(x)$ that runs from $x$ to $y$. Both $\gamma^{(x)} \circ \alpha$ and $f \circ \alpha$ are again geodesics, which must be equal by the Picard-Lindel\"of theorem, since $\gamma^{(x)}(x) = f(x)$ and $(\gamma^{(x)})^*_x = f^*_x$. In particular \begin{equation} f(y) = (f \circ \alpha)(1) = (\gamma^{(x)} \circ \alpha)(1) = \gamma^{(x)}(y), \end{equation} showing that $f|_{B_r(x)} = \gamma^{(x)}|_{B_r(x)}$. Now let $y \in X$ be any other point. Since $X$ is connected, we can choose a curve $c \subseteq \mathring{X}$ from $x$ to $y$, along with a sequence of points $x \eqqcolon x_0,x_1,\dots,x_j \coloneqq y$ on $c$ and open sets $U_i \ni x_i$ in $\mathring{X}$ such that \begin{itemize} \item $\gamma^{(x_i)}|_{U_i} = f|_{U_i}$, \item $U_i \cap U_{i+1} \neq \emptyset$ for all $0 \leq i \leq j-1$. \end{itemize} Therefore, we get $\gamma^{(x_i)}|_{U_i \cap U_{i+1}} = \gamma^{(x_{i+1})}|_{U_i \cap U_{i+1}}$ for all $0 \leq i \leq j-1$. Again, since global isometries on $\hyp^n$ agree everywhere if they agree on a non-empty open subset, we conclude that $\gamma^{(x)} = \gamma^{(x_0)} = \gamma^{(x_1)} = \dots = \gamma^{(x_j)} = \gamma^{(y)} = f$ on $\bigcup_{i=0}^j U_i$, finishing the proof. \end{proof} 
For any $R \geq 0$, we introduce the bundles \begin{align} & E^\rho_{M_R} \downarrow M_R \\ & E^\rho_{C_R} \downarrow C_R, \\ & E^\rho_{T_R} \downarrow T_R. \end{align} Here, $M_R$, $C_R$ and $T_R$ are the complete submanifolds of $\hyp^n$ as defined in Equations \ref{thick},\ref{thin} and \ref{torus}. As before, we let, for each $1 \leq j \leq k$, $C_0^j \cong [0,\infty) \times \reals^{n-1}$ be a connected component of $C_0$, so that for any $R \geq 0$, $C_R^j \coloneqq C_0^j \cap C_R$ and $T_R^j \coloneqq T_R \cap C_0^j$ are connected components of $C_R$, respectively $T_R$. 
\begin{lemma}\label{BUNDISO3} For each $1 \leq j \leq k$ and any $R \geq 0$, the collection of hyperbolic isometries \begin{align*} & f_R^j: C_R^j \cong [R,\infty) \times \reals^{n-1} \to C_0^j \cong [0,\infty) \times \reals^{n-1}, \\ & (t,x) \mapsto (t-R,e^{-R}x) \end{align*} extend to a flat bundle isometry \begin{align} F_R: E^\rho_{C_R} \downarrow C_R \to E^\rho_{C_0} \downarrow C_0, \end{align} which induces by restriction a flat bundle isometry \begin{align}F_R|_{E^\rho_{T_R}}\colon E^\rho_{T_R} \downarrow T_R \to E^\rho_{T_0} \downarrow T_0. \end{align} \end{lemma} 
\begin{proof}
By Lemma \ref{CONNECEQ}, there exists a unique hyperbolic isometry $\gamma_R^j \in G$ extending $f_R^j$ and a flat bundle isometry $F_R^j: E^\rho_{C_R^j} \downarrow C_R^j \to E^\rho_{C_0^j} \downarrow C_0^j$ of the form $F_R^j( (t,x),v) = ( \gamma_R^j. (t,x), \rho(\gamma_R^j) \cdot v)$. Notice that we have obvious identifications $\coprod_{j=1}^k \Gamma \times_{\Gamma_0^j} E^\rho_{C_0^j} \cong E^\rho_{C_0}$ and likewise $\coprod_{j=1}^k\Gamma \times_{\Gamma_0^j} E^\rho_{C_R^j} \cong E^\rho_{C_R}$ (as bundles over $C_0$, respectively $C_R$). Lastly, observe that the diffeomorphism \begin{align*} & F_R: \coprod_{j=1}^k \Gamma \times E^\rho_{C_R^j} \to \coprod_{j=1}^k  \Gamma \times E^\rho_{C_0^j}, \\
& \coprod_{j} \left( \gamma, (t,x),v \right) \mapsto \coprod_{j}\left( \gamma(\gamma_R^j)^{-1}, \gamma_R^j. (t,x), \rho(\gamma_R^j) \cdot v \right) \end{align*}
descends to a flat bundle isometry $F_R: \coprod_{j=1}^k\Gamma \times_{\Gamma_0^j} E^\rho_{C_R^j} \xrightarrow{\cong} \coprod_{j=1}^k\Gamma \times_{\Gamma_0^j} E^\rho_{C_0^j}$. The result follows.
\end{proof}
With respect to the hyperbolic metric $g$ and the canonical bundle metric $h_\rho$, we construct as in Section \ref{SecDeRha} the $L^2$-De Rham complex $\Omega_{(2)}^*(\hyp^n,E^\rho)$, as well as the three $L^2$-De Rham complexes
$\Omega_{(2)}^*(M_R,E^\rho_{M_R} )$, $\Omega_{(2)}^*(C_R,E^\rho_{C_R})$, $\Omega_{(2)}^*(T_R,E^\rho_{T_R})$ for each $R > 0$. Recall that all complexes come equipped with the linear, isometric $\Gamma$-action induced from the $G$-action on the respective bundles. Further, let 
\begin{align}& \Delta_*[E^\rho_{M_R}]\colon \Omega_{(2)}^*(M_R,E^\rho_{M_R} ) \to  \Omega_{(2)}^*(M_R,E^\rho_{M_R} ),  \\ 
%&\Delta[E^\rho_{C_R}] \colon \Omega_{(2)}^\bullet(C_R,E^\rho_{C_R}) \to \Omega_{(2)}^\bullet(C_R,E^\rho_{C_R}), \\
%& \Delta[E^\rho_{T_R}] \colon \Omega_{(2)}^\bullet(T_R,E^\rho_{T_R}) \to \Omega_{(2)}^\bullet(T_R,E^\rho_{T_R}), \\
& \Delta_*[E^\rho] \colon \Omega_{(2)}^*(\hyp^n,E^\rho) \to \Omega_{(2)}^*(\hyp^n,E^\rho). \end{align} 
be the (graded) Laplace operators of the respective bundles. Now
%\begin{proposition} \begin{enumerate} \item The complexes $\Omega_{(2)}(\hyp^n,E^\rho)$ and, for each $R >0$, $\Omega_{(2)}(M_R,E^\rho_{M_R})$ are Hilbert $%\vnN(\Gamma)$-cochain complexes. \item When regarded as unbounded, densely defined operators
 recall that $E^\rho \downarrow \hyp^n$ is $G$-equivariant homogeneous. Since $G = \Isom_0(\hyp^n)$ always contains a uniform lattice, it follows from the discussions of Section \ref{SecDeRha} that $\Delta_*[E^\rho]$ admits a self-adjoint minimal closure, and that $\Omega_{(2)}(\hyp^n,E^\rho)$ is a Hilbert $\vnN(\Gamma)$-cochain complex which is Fredholm. Similarily, since $E^\rho_{M_R} \downarrow M_R$ is $G_{M_R}$-equivariant by Lemma \ref{CONNECEQ} and $\Gamma / M_R$ is compact, the same statements hold true for the operator $\Delta_*[E^\rho_{M_R}]$ and the complex $\Omega_{(2)}^*(M_R,E^\rho_{M_R})$. From now on, we identify $\Delta_*[E^\rho_{M_R}]$ and $\Delta_*[E^\rho]$ with their respective minimal closure and consider the {\itshape heat traces} \begin{align} & \vntr(e^{-t\Delta_p[E^\rho]}) = \int_{\mathcal F} \tr (e^{-t\Delta_p[E^\rho]}(x,x))\; dx, \\ & \vntr(e^{-t\Delta_p[E^\rho_{M_R}]}) = \int_{\mathcal F_R} \tr (e^{-t\Delta_p[E^\rho_{M_R}]}(x,x)) \; dx.  \end{align} These are convergent integrals for each $t > 0$, whose respective values do not depend on the explicit choice of $\mathcal F$, respectively $\mathcal F_R$. \\ Observe that, since $E^\rho \downarrow \hyp^n$ is a homogeneous bundle, there is a smooth function $H_\rho(t): \reals_{>0} \to \reals$, satisfying \begin{equation}\label{PLANCHAREL} \tr(e^{-t\Delta_p[E^\rho]}(x,x)) \equiv H_\rho(t). \end{equation} Using the {\itshape Plancherel Formula}, $H_\rho(t)$ can actually be explicitly computed, as done in \cite[Section 9]{Muller:heat}. \\
We are now well-prepared to state the first key result of this paper, which we will obtain from Proposition \ref{l2nov1} and Corollary \ref{bddetclass} in the last section.
\begin{theorem}\label{detconvclass} The cochain complexes $\Omega^*(\hyp^n,E^\rho)$ and $\Omega^*(M_R,E^\rho_{M_R})$ for each $R >0$ are $L^2$-acyclic and have positive Novikov-Shubin invariants. In particular, the cochain complexes are of determinant class and we have \begin{align*} \int_{1}^\infty t^{-1} \vntr(e^{-t\Delta_*[E^\rho]}) dt < \infty, \end{align*} as well as \begin{align*} \int_{1}^\infty t^{-1} \vntr(e^{-t\Delta_*[E^\rho_{M_R}]}) dt < \infty \end{align*} for each $R >0$. The same holds true for the complexes $\Omega^*(\partial M_R,E^\rho_{\partial M_R})$ \end{theorem} The acyclicity statement from Theorem \ref{detconvclass}, together with Theorem \ref{asymptbdgeom} now imply: \begin{proposition}\label{trafo} The Hilbert $\vnN(\Gamma)$-cochain complexes $\Omega_{(2)}^*(\hyp^n,E^\rho)$ and $\Omega_{(2)}^*(M_R,E^\rho_{M_R})$ are $\zeta$-regular. That is, for each $0 \leq p \leq n$ and $s \in \mathbb C$ with $\Re(s) >> 0$, the integral expressions \begin{align} & \zeta_{p}(s) \coloneqq \Gamma(s)^{-1}\int_{0}^1 t^{s-1} \vntr(e^{-t\Delta_p[E^\rho]}) \;dt, \
\\ & \zeta_{p}^{R}(s) \coloneqq \Gamma(s)^{-1}\int_{0}^1 t^{s-1} \vntr(e^{-t\Delta_p[E^\rho_{M_R}]}) \;dt, \; \; \; R \geq 0 \end{align} determine holomorphic functions, each admitting meromorphic extensions on all of $\mathbb C$ which are regular at $0$. \end{proposition}
\begin{remark} The meromorphic extensions will also be denoted by $\zeta_p(s)$ and $\zeta_p^R(s)$, respectively. \end{remark} As a consequence of Theorems \ref{trafo} and \ref{detconvclass}, we can finally define the {\bfseries analytic $L^2$-torsion} $T_{(2)}^{An}(\Gamma \backslash M_R,\rho)$ and  $T_{(2)}^{An}(\Gamma \backslash \hyp^n,\rho)$ of the Hermitian bundles $E^\rho_{M_R} \downarrow M_R$ and $E^\rho \downarrow \hyp^n$ as \begin{align} &\log \left(T_{(2)}^{An}(\Gamma \backslash \hyp^n,\rho)\right) \coloneqq \sum_{p =0}^n \frac{p}{2} (-1)^{p+1} \left( \frac{d}{ds} \zeta_{p}(s)|_{s = 0} + \int_{1}^\infty t^{-1} \vntr(e^{-t\Delta_p[E^\rho]}) dt \right),
\\& \log \left(T_{(2)}^{An}(\Gamma \backslash M_R,\rho)\right) \coloneqq \sum_{p =0}^n \frac{p}{2} (-1)^{p+1} \left( \frac{d}{ds} \zeta_{p}^{R}(s)|_{s = 0} + \int_{1}^\infty t^{-1} \vntr(e^{-t\Delta_p[E^\rho_{M_R}]}) dt \right).\end{align} 
Observe that from \ref{PLANCHAREL}, it actually follows that there exists a number $\tau(\rho) \in \reals$ depending only on the representation $\rho$, such that for any lattice $\Gamma < \Isom^+(M,g)$, one has 
\begin{align*} \log \left(T_{(2)}^{An}(\Gamma \backslash \hyp^n,\rho)\right) = \Vol(\Gamma) \cdot \tau(\rho). \end{align*} For a detailed description of the element $\tau(\rho)$, we refer again to \cite[Section 9]{Muller:heat}.  The two main results of this paper, Theorems \ref{LTC} and \ref{STC}, can be summarized in one single statement: \begin{theorem}\label{anconv} For each $0 \leq p \leq n$, one has \begin{align}& \lim_{R  \to \infty} \frac{d}{ds} \zeta_{p}^{R}(s)|_{s = 0} = \frac{d}{ds} \zeta_p(s)|_{s = 0}, \\& \lim_{R \to \infty} \int_{1}^\infty t^{-1} \vntr(e^{-t\Delta_p[E^\rho_{M_R}]}) dt = \int_{1}^\infty t^{-1} \vntr(e^{-t\Delta_p[E^\rho]}) dt. \end{align}  In particular \begin{align} \lim_{R \to \infty} T_{(2)}^{An}(\Gamma \backslash M_R,\rho) = T_{(2)}^{An}(\Gamma \backslash \hyp^n,\rho). \end{align} \end{theorem} %For that purpose, we will first decompose the respective Mellin-Transforms as the formal sum of their {\itshape small} and {\itshape large} parts: \begin{align} & \zeta_{p}(s) = \overbrace{\Gamma(s)^{-1}\int_{0}^1 t^{s-1} \vntr(e^{-t\Delta_p[E^\rho]}) \;dt}^{ \eqqcolon \zeta_{k,sm}(s)} + \overbrace{\Gamma(s)^{-1}\int_{1}^\infty t^{s-1} \vntr(e^{-t\Delta_p[E^\rho]}) \;dt}^{=\colon\zeta_{p,la}(s)} \\
%& \zeta_{p}^{R}(s) \eqqcolon \overbrace{\Gamma(s)^{-1}\int_{0}^1 t^{s-1} \vntr(e^{-t\Delta_p[E^\rho_{M_R}]}) \;dt}^{ \eqqcolon \zeta_{k,sm}^{R}(s)} + \overbrace{\Gamma(s)^{-1}\int_{1}^\infty t^{s-1} \vntr(e^{-t\Delta_p[E^\rho_{M_R}]}) \;dt}^{\eqqcolon \zeta_{k,la}^{R}(s)}\end{align}
We will follow the strategy developed in \cite{Lueck:hyp} for the case of the trivial bundle (i.e.\ the bundle $E^{\rho} \downarrow \hyp^n$ associated with the trivial representation $\rho: G \to \ceals$) and show that it extends to the general case that we are concerned with here. The key results will be extracted from a thorough inspection of the asymptotic behavior of $\vntr(e^{-t\Delta_p[E^\rho]})$ and $\vntr(e^{-t\Delta_p[E^\rho_{M_R}]})$ for {\itshape small time} $t \to 0$, respectively for {\itshape large time} $t \to \infty$.% More precisely, we will show that the {\itshape small part} has a meromorphic extension that is holomorphic around $0$, while we will derive that the {\itshape large part} is actually an entire function ( We remark that the latter implies the commonly-referred-to {\itshape determinant class property} of the Laplacians $\Delta[E^\rho]$ and $\Delta[E^\rho_{M_R}]$, see for example \cite[Definition 1.3]{Schick:Det}).
 \\The respective methods involved in the inspection will actually be quite distinct, since the small time asymptotics depend only on the {\itshape local geometry} of $\hyp^n$, while for the large time asymptotics, the large scale geometry of the quotients $\Gamma \backslash M_R$, $\Gamma \backslash \hyp^n$ comes into play.\\ 
The limit formula from Theorem \ref{anconv} serves as the crucial ingredient in the proof of the final result of this paper: The $L^2$-Cheeger-M\"uller type theorem, showing the equality between the previously defined analytic $L^2$-torsion $T_{(2)}^{An}(\Gamma \backslash \hyp^n,\rho)$ and the topological $L^2$-torsion $T^{Top}_{(2)}(\Gamma \backslash \hyp^n,\rho)$, introduced in Section \ref{CombTopSec}. Of course, we still need to show that the latter quantity is well-defined. For that purpose, we observe that for each $R > 0$, any CW-structure on the compact manifold $\Gamma \backslash M_R$ serves as a finite CW-model for $\Gamma \backslash \hyp^n$. Furthermore, we regard $\rho$ as a representation of the common deck group $\Gamma$ of both $\hyp^n \to \Gamma \backslash \hyp^n$ and $M_R \to \Gamma \backslash M_R$. With this in mind, it is well-known that for $n$ odd-dimensional, the pair $(\Gamma \backslash \hyp^n,\rho)$ satisfies assertions $(T_1)$ and $(T_3)$ from Section \ref{CombTopSec} (See \cite[Lemma $4.3$]{Muller:Tor2} for $(T_1)$ and \cite[Lemma $0.12$]{Farrell:Wh} for a verification of $(T_3)$). To see that assertions $(T_2)$ and $(T_4)$ are also satisfied, we compute for each $0 \leq p \leq n$
\begin{align*}& b_{k,(2)}^{Top}(\Gamma \backslash \hyp^n,\rho) \stackrel{\text{Definition \ref{TopBetDef}}}{=} b_{k,(2)}^{Top}(\Gamma \backslash M_R,\rho) \stackrel{\text{Corollary \ref{DeRhCor}}}{=} b_{k,(2)}^{An}(\Gamma \backslash M_R,\rho) \stackrel{\text{Theorem \ref{detconvclass}}}{=} 0 \\& \alpha_p^{Top}(\Gamma \backslash \hyp^n,\rho) \stackrel{\text{Definition \ref{TopBetDef}}}{=} \alpha_p^{Top}(\Gamma \backslash M_R,\rho) \stackrel{\text{Corollary \ref{DeRhCor}}}{=} \alpha_p^{An}(\Gamma \backslash M_R,\rho) \stackrel{\text{Theorem \ref{detconvclass}}}{>} 0. \end{align*} We conclude that $T^{Top}_{(2)}(\Gamma \backslash \hyp^n,\rho)$ is well-defined with \begin{equation} T^{Top}_{(2)}(\Gamma \backslash \hyp^n,\rho) = T^{Top}_{(2)}(\Gamma \backslash M_R,\rho) \end{equation} for any $R > 0$. \\
As stated above, our ultimate goal is to show the equality of the analyic and topological torsion on $\Gamma \backslash \hyp^n$, namely \begin{theorem}[Hyperbolic Cheeger-M\"uller Theorem]\label{HypCheMül} One has \begin{equation} T^{Top}_{(2)}(\Gamma \backslash \hyp^n,\rho) = T_{(2)}^{An}(\Gamma \backslash \hyp^n,\rho). \end{equation} \end{theorem} This theorem is a strict generalization of the main result in \cite{Lueck:hyp}, where they were able to prove it for the case of the trivial representation $\rho \equiv \unit$. \\
We now explain how to derive Theorem \ref{HypCheMül} from Theorem \ref{anconv}. To do so, we need two additional results, the first of which is a Cheeger-M\"uller type Theorem on the submanifold $M_R$. Namely, we may apply Theorem \ref{CheMulComp} to the bundle $E^\rho_{M_R} \downarrow M_R$ because \begin{itemize} \item the relevant determinant class conditions are satisfied by Theorem \ref{detconvclass}, \item the canonical bundle metric $h_\rho$ is unimodular (see \cite[Lemma $5.5.1$]{Ich}), \end{itemize} which yields
%\begin{remark} The quantities $T_{(2)}^{An}(\Gamma \backslash M_R,\rho)$ and $T_{(2)}^{An}(\Gamma \backslash \hyp^n,\rho)$ both depend on the pair of metrics $(g,h_\rho)$, which is why we will often write $T_{(2)}^{An}(\Gamma \backslash \hyp^n,\rho,g,h_\rho)$, respectively $T_{(2)}^{An}(\Gamma \backslash \hyp^n,\rho,g,h_\rho)$, to emphasize this dependency. \end{remark}
\begin{theorem}\label{TorQuot} One has \begin{equation} \log \left( \frac{T_{(2)}^{An}(\Gamma \backslash M_R,\rho)}{T^{Top}_{(2)}(\Gamma \backslash M_R,\rho)} \right) = \dim_{\ceals}(\rho) \cdot \log \left( \frac{T_{(2)}^{An}(\Gamma \backslash M_R,\unit)}{T^{Top}_{(2)}(\Gamma \backslash M_R,\unit)} \right). \end{equation} \end{theorem}
By a careful analysis of contribution to the anomaly $\log \left(\frac{T_{(2)}^{An}(\Gamma \backslash M_R,\unit)}{T^{Top}_{(2)}(\Gamma \backslash M_R,\unit)}\right)$ coming from the boundary $\partial M_R$, together with the fact that $\lim_{R \to \infty} \Vol_{g_{\hyp^n}}(\Gamma \backslash \partial M_R) = 0$, L\"uck and Schick were able to prove an asymptotic vanishing of the anomaly, see \cite[Section $1$]{Lueck:hyp} for more details:
\begin{theorem}\label{AnZero}
One has \begin{equation} \lim_{R \to \infty} \log \left(\frac{T_{(2)}^{An}(\Gamma \backslash M_R,\unit)}{T^{Top}_{(2)}(\Gamma \backslash M_R,\unit)}\right) = 0. \end{equation}
\end{theorem} 
{\bfseries Proof of Theorem \ref{HypCheMül}}.
For all $R > 0$, it holds that
\begin{align*} & \log(T^{Top}_{(2)}(\Gamma \backslash \hyp^n,\rho)) = \log(T^{Top}_{(2)}(\Gamma \backslash M_R,\rho)) \\ & =
\log(T_{(2)}^{An}(\Gamma \backslash M_R,\rho)) + \log \left( \frac{T^{Top}_{(2)}(\Gamma \backslash M_R,\rho)}{T_{(2)}^{An}(\Gamma \backslash M_R,\rho)} \right) \\ & \stackrel{\text{Theorem} \ref{TorQuot}}{=} \log(T_{(2)}^{An}(\Gamma \backslash M_R,\rho)) + \dim(\rho) \cdot \log \left( \frac{T^{Top}_{(2)}(\Gamma \backslash M_R,\unit)}{T_{(2)}^{An}(\Gamma \backslash M_R,\unit)} \right) . \end{align*}
Passing to the limit $R \to \infty$, the result then follows from Theorem \ref{anconv} and Theorem \ref{AnZero}. \qed

\section{Small-time convergence}

Throughout the whole section, we assume that the $L^2$-acyclicity statement from Theorem \ref{detconvclass} has already been proven. This will cause us no trouble, since Theorem \ref{detconvclass} will be proven in the last section completely independently from the results shown here. \\
Recall the following notions from the previous section: Let $\hyp^n$ be hyperbolic $n$-space, $G = \Isom_0(\hyp^n)$ the identity component of isometries on $\hyp^n$ and \begin{equation} \Gamma \subset G \end{equation} a torsion-free, non-uniform lattice. Let $\rho: G \to GL(V)$ be a complex, finite-dimensional irreducible representation of $G$ and let $E^\rho \downarrow \hyp^n$ be the associated flat bundle over $\hyp^n$, equipped with the canonical $G$-equivariant metric $h^\rho$. Further, let $C_R, T_R$ and $M_R$ be the submanifolds defined at the beginning of the previous section for each $R > 0$. Recall the associated, $\Gamma$-invariant bundles \begin{align*} & E^\rho_{M_R} \downarrow M_R, \\ & E^\rho_{C_R} \downarrow C_R, \\ & E^\rho_{T_R} \downarrow T_R, \end{align*} obtained by restriction of $E^\rho$ to the respective base space. A central fact that many important conclusions drawn from this section are based upon lies in the next result.
\begin{lemma}\label{LOCHOMIMP} For any $R > 0$, the bundle $E^\rho_{M_R} \downarrow M_R$ is locally homogeneous. In fact, \begin{enumerate} \item $E^\rho_{\mathring{M}_R} \downarrow \mathring{M}_R$ has the same local isometry type as $E^\rho \downarrow \hyp^n$. \item For any $R > 1$, there exists a neighborhood $U_R$ of $\partial M_R$, such that the collection of restricted bundles $(E_{U_R} \downarrow U_R)_{R > 1}$ all of the same local isometry type. \end{enumerate} \end{lemma}
\begin{proof}
Because of Lemma \ref{CONNECEQ}, it suffices to prove the analogous statements about the local isometry type and local homogenity for the underlying base spaces of the bundles.
With this in mind, assertion $1$ follows because $\hyp^n$ is a homogeneous space, which is why all open submanifolds, in particular also $\mathring{M}_R$, have the same local isometry type as $\hyp^n$. 
To verify assertion $2$, we choose $U_R \cong T_{R-1}$. All components $T_{R-1}^j \equiv [R-1,R] \times \reals^{n-1} $ of $T_R$ are pairwise isometric. Moreover, given any two points $(R,x),(R,y) \in \{R\} \times \reals^{n-1} = T_{R-1}^j \cap \partial M_R$ within a given component of $\partial M_R$, we can choose a Euclidean isometry $f \in \Isom^+(\reals^{n-1})$ sending $x$ to $y$. Evidently, this induces an isometry $\unit_{[R-1,R]} \times f$ on $T_{R-1}$ with which we ``move vertically'', fixing each component of $T_{R-1}$ and sending $(R,x)$ to $(R,y)$. Also, observe that the isometries from Lemma \ref{BUNDISO3} show that we can transport all $T_{R-1}$ for $R > 1$ isometrically into one another by ``moving horizontally''. All in all, this shows that collection of $T_{R-1}$, $R >1$ all have the same local isometry type, verifying assertion $2$ and thus finishing the proof. 
\end{proof}
Consider the $\Gamma$-regularized traces 
\begin{align} \vntr(e^{-t\Delta_p[E^\rho]}) = \int_{\mathcal F} \tr(e^{-t\Delta_p[E^\rho]}(x,x)) dx,\\
 \vntr(e^{-t\Delta_p[E^\rho_{M_R}]}) = \int_{\mathcal F_R}\tr(e^{-t\Delta_p[E^\rho_{M_R}]}(x,x)) dx, \end{align} where as before, $dx$ denotes the volume form on $\hyp^n$ induced by the hyperbolic metric $g$. 
Here, we can and have chosen the $\Gamma$-fundamental domains $\mathcal F$, respectively $\mathcal F_R$ for the $\Gamma$-action on $\hyp^n$, respectively ${M}_R$, so that for each $R > 0$ \begin{enumerate} \item ${M}_R \cap \mathcal F = \mathcal F_R$, so that $\mathcal F =\bigcup_{R \geq 0} \mathcal F_R$. 
\item $\partial {M}_R \cap \mathcal F =\partial \mathcal {F}_R$ is a fundamental domain for the $\Gamma$-action on $\partial {M}_R$.
\item There exists a finite family $(\mathcal G_j)_{j = 1}^k$ with each $\mathcal G_j \subseteq \reals^{n-1}$ a compact euclidean submanifold, such that for $0 \leq R < S < \infty$, we have \begin{enumerate} \item $\mathcal F_S \setminus \mathcal F_R = [R,S] \times \coprod_{j  = 1}^k \mathcal G_j$,  
\item $\partial \mathcal F_R = \{R\} \times \coprod_{j = 1}^k \mathcal G_j$. 
\end{enumerate}
\end{enumerate}
Now recall again that the Hermitian bundles $E^\rho \downarrow \hyp^n$ and $E^\rho_{M_R} \downarrow M_R$ satisfy assumptions $(A_1)$, respectively $(A_2)$ from Section \ref{SecDeRha}. Lemma \ref{LOCHOMIMP}, together with the homogenity formula from Corollary \ref{IMPORTANTE} and the asymptotic expansions outlined in Theorem \ref{asymptbdgeom} readily imply:
\begin{proposition}\label{asyR}
For fixed $p \in \mathbb N$ and each $i = 0,\dots,n$, there exist {\bfseries constants} $a_i,b_i \in \mathbb C$, such that for $t \to 0$, we have 
\begin{align}\vntr(e^{-t\Delta_p[E^\rho]}) = \Vol(\mathcal F) \sum_{i=0}^n t^{-(n-i)/2} a_i + \mathcal O(t^{1/2}), \label{asy1}\end{align} and, for each $R > 0$, we have for $t \to 0$
 \begin{align} 
\vntr(e^{-t\Delta_p[E^\rho_{M_R}]}) = \sum_{i=0}^n t^{-(n-i)/2} ( \Vol(\mathcal{F}_R) a_i + \Vol(\partial \mathcal F_R)  b_i )+ \mathcal O(t^{1/2}). \label{asy2} \end{align}
\end{proposition}

In order to streamline the notation, we will set \begin{align} \alpha_i \coloneqq \Vol(\mathcal F) a_i \; \; , \; \; \alpha_i^{R} \coloneqq \Vol(\mathcal F_R) a_i \; \;  , \; \; \beta_i^{R} \coloneqq \Vol(\partial \mathcal F_R) b_i. \end{align} whenever $p = 0, \dots, n$ is clear from the context. Combining Proposition \ref{asyR} with Corollary \ref{zetaregi}, we now get

\begin{corollary}\label{zetareg} It holds that
\begin{align}
& \frac{d}{ds} \zeta_p^R(s)|_{s=0}  =  \int_{0}^{1} \left( \vntr(e^{-t \Delta_p[E^\rho_{M_R}]}) - \sum_{i=0}^n t^{-(n-i)/2} (\alpha_i^R+ \beta_i^R) \right) \: \frac{dt}{t} \nonumber \\ &+ \sum_{i = 0}^n c(i,n) (\alpha_i^R   +  \beta_i^R),\\
& \frac{d}{ds} \zeta_p(s)|_{s=0} \coloneqq \int_{0}^{1}  \left( \vntr(e^{-t\Delta_p[E^\rho]}) - \sum_{i=0}^n t^{-(n-i)/2} \alpha_i \right)   \: \frac{dt}{t} + \sum_{i=0}^n c(i,n) \alpha_i,
\end{align}
where 
\begin{align} c(i,n) \coloneqq \begin{cases} - \frac{n -i}{2} & i \neq n, \\ \frac{d\Gamma}{ds} |_{s = 1} & i = n. \end{cases} \end{align}
\end{corollary}

Since $\tr(e^{-t\Delta_p}(x,x)) = H_\rho(t)$ is constant on $\mathbb H^n$, we obtain from Equation \eqref{asy1} the following:

 \begin{corollary} \label{uni} There exists a constant $C > 0$, such that for all $x \in \mathbb H^n$, we have for $t \to 0$ \begin{equation} |\tr(e^{-t\Delta_p[E^\rho]}(x,x)) - \sum_{i=0}^n t^{-(n-i)/2}a_i| \leq Ct^{1/2}. \end{equation} \end{corollary}

As a last intermediate result, we need: 

\begin{corollary}\label{part} For $R \geq 1$ and $t \to 0$, we have \begin{align}
& \int_{\mathcal F_R \setminus \mathcal F_{R-1}} \tr(e^{-t\Delta_p[E^\rho_{M_R}]}(x,x)) \; dx  - \sum_{i=0}^n t^{-(n-i)/2} ( \Vol(\mathcal F_R \setminus \mathcal F_{R-1}) a_i + \Vol(\partial \mathcal F_R) b_i ) \nonumber \\ &  \in \mathcal O(t^{1/2}). \end{align}  \end{corollary}

\begin{proof} We can write \begin{align} & \int_{\mathcal F_R \setminus \mathcal F_{R-1}} \tr(e^{-t\Delta_p[E^\rho_{M_R}]}(x,x)) \; dx  - \sum_{i=0}^n t^{-(n-i)/2} ( \Vol(\mathcal F_R \setminus \mathcal F_{R-1}) a_i + \Vol(\partial \mathcal F_R) b_i )  \nonumber \\
& = \vntr(e^{-t\Delta_p[E^\rho_{M_R}]}) - \sum_{i=0}^n t^{-(n-i)/2} (\Vol(\mathcal F_R)a_i + \Vol(\mathcal \partial F_R)b_i) \label{bd1} \\ & - \int_{\mathcal F_{R-1}} \tr(e^{-t\Delta_p[E^\rho_{M_R}]}(x,x)) - \sum_{i=0}^n t^{-(n-i)/2} a_i \; dx. \label{bd3} \end{align}
 \ref{bd1} is in $\mathcal O(t^{1/2})$ by Corollary \ref{asyR}. Next, observe that $d_{{M}_R}(x) \geq 1$ for any $x \in \mathcal F_{R-1}$. Therefore, applying Theorem \ref{heat} to $M = \hyp^n$, $N = {M}_R$ and $D =1$, we find constants $C,\kappa > 0$, such that for 
all $x \in \mathcal F_{R-1}$, we have \begin{equation} |\tr(e^{-t\Delta_p[E^\rho_{M_R}]} - e^{-t\Delta_p[E^\rho]})(x,x)| < Ce^{-2/(\kappa t)} \end{equation} Since $e^{-2/(\kappa t)} \in \mathcal O(t^{1/2})$, we can apply Corollary \ref{uni} and obtain that \ref{bd3} is also in $\mathcal O(t^{1/2})$. The result follows. \end{proof}

We are now equipped with all the tools needed to prove the main theorem of this section, the {\itshape small-time convergence} of the analytic $L^2$-torsion:

\begin{theorem}[Small-time convergence]\label{STC} For each $0 \leq p \leq n$, we have \begin{equation} \lim_{R \to \infty} \frac{d}{ds} \zeta_p^R(s)|_{s=0} = \frac{d}{ds} \zeta_p(s)|_{s=0}  . \end{equation} \end{theorem}

\begin{proof} 

First, observe that $\lim_{R \to \infty} \Vol(\mathcal F_R) = \Vol(\mathcal F)$ and, since $\partial \mathcal F_R$ is a flat subspace of $\partial {M}_R \cong \mathbb R^{n-1}$ which is equipped with the flat metric $e^{-2R}dx$, we also have $\lim_{R \to \infty} \Vol(\partial \mathcal F_R) = 0$. We deduce that $\lim_{R \to \infty} (\alpha_i - \alpha_i^{R}) =  \lim_{R \to \infty} \beta_i^{R} = 0$ for each $i = 0,\dots,n$. Therefore, the statement of the theorem will follow once we show that for 
\begin{equation} \Sigma[R] \coloneqq  \int_{0}^{1} \vntr(e^{-t \Delta_p[E^\rho]}) - \vntr(e^{- t\Delta_p[E^\rho_{M_R}]}) - \sum_{i=0}^n t^{-(n-i)/2} (\alpha_i - \alpha_i^{R} - \beta_i^{R}) \frac{dt}{t}, \end{equation}
we have 
\begin{equation} \lim_{R \to \infty}  \Sigma[R]  = 0. \end{equation}
Recall that $\vntr(e^{-t\Delta_p[E^\rho]}) = \Vol(\mathcal F) \tr(e^{-t\Delta_p[E^\rho]}(x,x))$ for any $x \in \mathcal F$. 
For fixed $R > 2$, we can therefore decompose $\Sigma[R] = \Sigma_1[R] + \Sigma_2[R] + \Sigma_3[R] - \Sigma_4[R]$, with  
\begin{align} & \Sigma_1[R] \coloneqq \Vol(\mathcal F \setminus \mathcal F_{R-1})  \int_{0}^{1} \tr(e^{-t\Delta_p[E^\rho]}(x,x)) - \sum_{i=0}^n t^{-(n-i)/2} \alpha_i \frac{d}{dt}, \\
&\Sigma_2[R] \coloneqq \int_{0}^{1} \int_{\mathcal F_{R/2}} \tr((e^{-t\Delta_p[E^\rho]} - e^{-t\Delta_p[E^\rho_{M_R}]})(x,x)) dx \frac{d}{dt}, \\
&\Sigma_3[R] \coloneqq \int_{0}^{1} \int_{\mathcal F_{R-1} \setminus \mathcal F_{R/2}} \tr((e^{-t\Delta_p[E^\rho]} - e^{-t\Delta_p[E^\rho_{M_R}]})(x,x)) dx \frac{d}{dt}, \\
&\Sigma_4[R] \coloneqq \int_{0}^{1} \int_{\mathcal F_R \setminus \mathcal F_{R-1}} \tr((e^{-\Delta_p[E^\rho_{M_R}]})(x,x)) dx  \\ & - \sum_{i=0}^n t^{-(n-i)/2} (\Vol(\mathcal F_R \setminus \mathcal F_{R-1}) a_i + \Vol(\partial \mathcal F_R) b_i) \frac{dt}{t}.
\end{align}
This splitting is allowed, since each one of these integrals converges, as shown in the course of the proof of the fact that $\lim_{R \to \infty} \Sigma_i[R] = 0$ for each $i = 1,\dots,4$. For this purpose, observe first that we may apply Corollary \ref{uni} to obtain a constant $C > 0$ independent of $R$, such that 
\begin{equation} |\Sigma_1[R]| \leq C\Vol(\mathcal F \setminus \mathcal F_{R-1}) \int_{0}^{1} t^{-1/2} dt = 2 C \Vol(\mathcal F \setminus \mathcal F_{R-1}) \xrightarrow{R \to \infty} 0. \end{equation} 
Secondly, observe that for any $R > 0$, the bundle map \begin{equation} \tr: \pi_1^*(E^*) \otimes \pi_2^*(E) \downarrow {M}_R \times {M}_R \to \mathbb C \times {M}_R \times {M}_R \downarrow {M}_R \times {M}_R \end{equation} is uniformly bounded (with respect to the canonical constant Hermitian metric on $\mathbb C \times {M}_R \times {M}_R \downarrow {M}_R \times {M}_R$) by a constant independent of $R$. 
Now $d_{\mathcal F_{R}}(x) \geq 1$ for any $x \in \mathcal F_{R-1}$ and $d_{\mathcal F_{R}}(x) \geq R/2 > 1$ for any $x \in \mathcal F_{R/2}$. Thus, we can now apply Theorem \ref{heat}.$(2)$ with $D = 1$ and obtain constants $C_1,C_2$, independent of $R$, such that \begin{align}
&| \tr((e^{-t\Delta_p[E^\rho]} - e^{-t\Delta_p[E^\rho_{M_R}]})(x,x))| < C_1 e^{-R^2/(2tC_2)}  \; \; \; \; \forall x \in \mathcal F_{R/2}, \\
& | \tr((e^{-t\Delta_p[E^\rho]} - e^{-t\Delta_p[E^\rho_{M_R}]})(x,x))|  < C_1 e^{-2/(tC_2)} \; \; \; \; \forall x \in \mathcal F_{R-1}. \end{align}
Heuristically, these estimates imply that the integrand in $\Sigma_2[R]$ decays exponentially fast in $R$, while the volume of the domain of integration $\mathcal F_{R/2}$ is, of course, uniformly bounded by the volume of $\mathcal F$. On the other hand, the integrand in $\Sigma_3[R]$ is uniformly bounded for all $R$, while the volume of the domain of integration, $\mathcal F_{R-1} \setminus \mathcal F_{R/2}$ decays exponentially fast. Put into action, we obtain 
\begin{align} & |\Sigma_2[R]| \leq C_1 \Vol(\mathcal {F}_{R/2}) \int_{0}^1 e^{-R^2/(2tC_2)} \frac{dt}{t}  \leq C_1\Vol(\mathcal F_{R/2})  \int_{0}^1 e^{-R^2/(2tC_2)} \frac{dt}{t^2} \\  \nonumber & \leq  \Vol(\mathcal F) \frac{2C_1C_2}{R^2} e^{-R^2/(2C_2)} \xrightarrow{R \to \infty} 0, \\
& |\Sigma_3[R]| \leq C_1 \Vol(\mathcal {F}_{R-1} \setminus \mathcal F_{R/2}) \int_{0}^{1} e^{-2/(tC_2)} \frac{dt}{t} \leq \frac{C_1C_2}{2} \Vol(\mathcal F_{R-1} \setminus \mathcal F_{R/2}) e^{-2/C_2} \\ \nonumber &
\leq \frac{C_1C_2}{2} e^{-(2+R)/C_2}(R/2 - 1) \xrightarrow{R \to \infty} 0.   \end{align}
The proof of $\lim_{R \to \infty} |\Sigma_4[R]| = 0$ requires a little more work. For that purpose, first define the horoball $\hyp^n_R \coloneqq (- \infty, R] \times \reals^{n-1} \subset \hyp^n$, so that the restriction of the hyperbolic metric on $\hyp^n_R$ is of the warped product form $dr^2 + e^{-2r}dx^2$. Denote by
$\Delta_p[\hyp^n_R]$ the Bochner-Laplace operator corresponding to the flat restriction bundle $E^{\rho}|_{\hyp^n_R}$ with corresponding heat kernel $e^{-t\Delta_p[\hyp^n_R]}(x,y)$. %, for $i = 0,\dots,p$, $\alpha_i^R: \hyp^n_R \to \ceals $ and $\beta_i^R: \partial \hyp^n_R \to \ceals$ the respective density functions. 
This allows us to decompose $\Sigma_4[R]$ as the sum of the following convergent integrals
\begin{align} &\Sigma_4[R] = \overbrace{\int_{0}^{1} \int_{\mathcal F_R \setminus \mathcal F_{R-1}} \tr((e^{-t\Delta_p[E^\rho_{M_R}]} - e^{-t\Delta_p[{M}_R \setminus {M}_{R-2}]})(x,x)) dx \frac{dt}{t}}^{\Sigma_{4,1}[R]},  \\ 
&+ \overbrace{\int_{0}^{1} \int_{\mathcal F_R \setminus \mathcal F_{R-1}} \tr((e^{-t\Delta_p[{M}_R \setminus {M}_{R-2}]} - e^{-t\Delta_p[\hyp^n_R]})(x,x)) dx\frac{dt}{t}}^{\Sigma_{4,2}[R]}, & \\
& + \overbrace{\int_{0}^{1} \int_{\mathcal F_R \setminus \mathcal F_{R-1}} \tr(e^{-t\Delta_p[\hyp^n_R]}(x,x)) dx- \sum_{i=0}^n t^{-(n-i)/2} (\Vol(\mathcal F_R \setminus \mathcal F_{R-1}) a_i + \Vol(\partial \mathcal F_R) b_i) \frac{dt}{t}}^{\Sigma_{4,3}[R]}. 
 \end{align}
The first two summands converge, as is shown in the course of the next paragraph, and therefore also the third. To deal with the first two summands, we can apply Theorem \ref{heat} in the following fashion: For $\Sigma_{4,1}[R]$, we put ${M}_R$ in the role of the ambient manifold and regard ${M}_R \setminus {M}_{R-2}$ as a submanifold. In this setting, we have
$d_{{M}_R \setminus {M}_{R-2}}(x) \geq 1$ for any $x \in \mathcal F_R \setminus \mathcal F_{R-1}$. Therefore, Theorem  \ref{heat} provides us with constants $D_1,D_2 > 0$, such that \begin{equation}
|\tr((e^{-t\Delta_p[E^\rho_{M_R}]} - e^{-t\Delta_p[{M}_R \setminus {M}_{R-2}]})(x,x)) | < D_1 e^{\frac{2}{D_2t}}. \end{equation}
Moreover, as ${M}_R$ and ${M}_0$ have isometric neighborhoods (in $\hyp^n$), both $D_1$ and $D_2$ can be chosen independently of $R$. This implies that 
\begin{equation} |\Sigma_{4,1}[R]| \leq \Vol(\mathcal F_{R} \setminus \mathcal F_{R-1}) \frac{D_1D_2}{2} e^{\frac{2}{D_2}} \xrightarrow{R \to \infty} 0. \label{1st} \end{equation} 
Replacing ${M}_R$ by $\hyp^n_R$ in the role of the ambient manifold, an analogous argument yields 
\begin{equation} \lim_{R \to \infty} |\Sigma_{4,2}[R]| = 0. \label{2nd} \end{equation} 
Next, observe that by our choice of $\mathcal F_R$, we have \begin{align} & \Vol(\mathcal F_R \setminus \mathcal F_{R-1}) = e^{-2R +2} \Vol(\mathcal F_1 \setminus \mathcal F_0),  \label{bound1} \\ & \Vol(\partial \mathcal F_R) = e^{-2R + 2} \Vol(\partial \mathcal F_1). \label{bound2} \end{align}
Secondly, recall that $\reals^{n-1}$ has transitive isometry group, and observe that any isometry $i: \reals^{n-1} \to \reals^{n-1}$ extends to an isometry $\unit_{(-\infty,R]}\times i: \hyp^n_R \to \hyp^n_R$. Moreover, we have an isometry 
\begin{align} & I_R: \hyp^n_R \to \hyp^n_1,  \nonumber \\ & I_R(u,y) \coloneqq (u + 1 - R, e^{R-1}y). \end{align} Using Lemma \ref{CONNECEQ} and Equation \ref{equivariance}, we therefore obtain a smooth map $h: (0,\infty) \times (-\infty,1] \to \mathbb C$, such that for $x = (u,v) \in (-\infty,\infty) \times \reals^{n-1}$, we have
\begin{align}& \tr(e^{-t\Delta[\hyp^n_1]}(x,x)) = h(t,u), & \text{if} \; x \in \hyp^n_1, \\ 
& \tr(e^{-t\Delta[\hyp^n_R]}(x,x)) = h(t,u + 1 - R), & \text{if} \; x\in \hyp^n_R . \end{align}
This implies that
\begin{align}&  \int_{\mathcal{F}_R \setminus \mathcal F_{R-1}} \tr(e^{-t\Delta[\hyp^n_R]}(x,x)) dx = \int_{\mathcal G} \int_{R-1}^{R} h(t,u -R +1) e^{-2u} du dy  \nonumber\\ 
& = \int_{\mathcal G} \int_{0}^{1} h(t,u) e^{-2(u + R -1)} du dy = e^{-2R + 2} \int_{\mathcal G} \int_{0}^{1} h(t,u) e^{-2u} du dy \nonumber\\ 
& = e^{-2R + 2} \int_{\mathcal F_1 \setminus \mathcal F_0} \tr(e^{-t\Delta[\hyp^n_1]})(x,x) dx. \label{bound3}  \end{align}
Equations \ref{bound1}, \ref{bound2} and \ref{bound3} now yield the equality
\begin{align} & \Sigma_{4,3}[R]  \nonumber = e^{-2R + 2}  \bigg( \int_{0}^{1} \int_{\mathcal F_1 \setminus \mathcal F_0} \tr(e^{-t\Delta_p[\hyp^n_1]}(x,x)) )dx \nonumber \\&\qquad  - \sum_{i=0}^n t^{-(n-i)/2} (\Vol(\mathcal F_1\setminus \mathcal F_{0}) a_i + \Vol(\partial \mathcal F_1) b_i) \frac{dt}{t} \bigg)  \xrightarrow{R \to \infty} 0. \label{3rd} \end{align}
Therefore, $\lim_{R \to \infty} \Sigma_{4}[R] = 0$ follows from Equations \ref{1st}, \ref{2nd} and \ref{3rd} and Corollary \ref{part}, finally finishing the proof of the theorem. 
\end{proof}

\section{Large-time convergence} 
The main result of this section relies on a careful analysis of the spectral density functions associated with different de Rham-type complexes derived from a given flat bundle. We first establish general estimates in this vein before applying them to our special hyperbolic scenario.
\subsection{Sobolev complexes}
We assume the same setup as in Section \ref{SecDeRha}. 
For the Riemannian manifold $(M,g)$ under inspection, we suppose first that $\partial M = \emptyset$. Then we define for each $s \in \mathbb R_{\geq 0}$ the {\itshape $s$-th Sobolev space} $\sob_s^*(M,E)$ to be the minimal domain of the unbounded operator $(1 + \Delta_*)^{s/2}$ (constructed by means of the spectral theorem) defined over $\Omega^*_{(2)}(M,E)$. It becomes a Hilbert space when equipped with the inner product
\begin{equation} \langle \omega, \sigma \rangle_{s} \coloneqq \langle (1 + \Delta_*)^{s/2} \omega, \sigma \rangle. \end{equation}
Now suppose that $\partial M \neq \emptyset$. Using that $\partial \partial M = \emptyset$, we define for each $k \in \mathbb N_0$ the Sobolev spaces $\sob_k^*(M,E)$ (with absolute boundary conditions) inductively as the completion of $\Omega^*_c(M,E)$ with respect to the norms 
\begin{align*}& ||\omega||_0^2 = ||\omega||^2, \\& ||\omega||_{k+1}^2 \coloneqq ||\omega||_k^2 + ||d\omega||_k^2 + ||\delta \omega||_k^2 + ||i^*\#\omega||_{k + 1/2}^2. \end{align*}
Since all operators involved in the construction of the Sobolev norms are $\Gamma$-equivariant, one can show in a straightforward fashion that each space $\sob_s^*(M,E)$ is a Hilbert $\vnN(\Gamma)$-module, $\Gamma$-equivariantly isomorphic to $\sob_s(\mathcal F, E|_{\mathcal F}) \hat \otimes L^2(\Gamma)$ with $\mathcal F \subseteq M$ a $\Gamma$-fundamental domain, so that $\partial \mathcal F \subseteq \partial M$ a fundamental domain for the $\Gamma$ action on $\partial M$ (just as in the case $s = 0$ seen previously). \\
Similarly as in the empty boundary case, we could also define $\sob_k^*(M,E)$ as the minimal domains of appropriate unbounded operators (See \cite[Theorem $3.4.4$]{Ich}). This way, it becomes apparent that we have an inclusion of subspaces $... \subseteq \sob_2^*(M,E) \subseteq \sob_1^*(M,E) \subseteq \sob_0^*(M,E) = \Omega_{(2)}^*(M,E)$.
\begin{dfn}[The de Rham and Sobolev complexes] Let $E \downarrow M$ be as above and let $0 \leq p \leq n$. 
\begin{enumerate} \item The {\bfseries Sobolev chain complex at level $p$}, denoted by $D_p[E]$ is the cochain complex of Hilbert spaces, defined as \begin{equation}
\dots \rightarrow 0 \rightarrow \sob_2^{p-1}(M,E) \xrightarrow{d} \sob_1^{p}(M,E) \xrightarrow{d} \sob_0^{p+1}(M,E) = \Omega_{(2)}^{p+1}(M,E) \rightarrow 0 \rightarrow \dots. \end{equation}
\item The {\bfseries absolute Sobolev chain complex at level $p$}, denoted by $D_{p,abs}[E]$ is the cochain complex of Hilbert spaces, defined as \begin{equation}
\dots \rightarrow 0 \rightarrow \sob_{2,abs}^{p-1}(M,E) \xrightarrow{d} \sob_{1,abs}^{p}(M,E) \xrightarrow{d} \sob_0^{p+1}(M,E) \rightarrow 0 \rightarrow \dots, \end{equation} where
\begin{align} & \sob_{2,abs}^{p-1}(M,E) \coloneqq  \{ \omega \in  \sob_2^{p-1}(M,E): i^*(\#\omega) = 0  = i^*(\# d\omega) \}, \\
& \sob_{1,abs}^{p}(M,E) \coloneqq \{ \omega \in \sob_1^{p}(M,E) : i^*(\#\omega) = 0  \}. \end{align} and the inner product on each space is the one induced by $\sob_2^{p-1}(M,E)$, respectively $\sob_1^p(M,E)$. 
\item The {\bfseries de Rham complex at level $p$}, denoted by $L_{p}[E]$ is the cochain complex of Hilbert spaces, defined as \begin{equation}
\dots \rightarrow 0 \rightarrow \Omega_{(2)}^{p-1}(M,E) \xrightarrow{d} \Omega_{(2)}^p(M,E) \xrightarrow{d} \Omega_{(2)}^{p+1}(M,E) \rightarrow 0 \rightarrow \dots. \end{equation}
\end{enumerate}
\end{dfn}
All three complexes are Hilbert $\vnN(\Gamma)$-cochain complexes. We have an inclusion of cochain complexes $D_{p,abs}[E] \subseteq D_{p}[E] \subset L_p[E] \subseteq \Omega_{(2)}^*(M,E)$ for any $0 \leq p \leq n$, so that, with the previously defined norms, $D_{p,abs}[E] \subseteq D_{p}[E]$ is a closed subcomplex and the inclusion $D_p[E] \subset L_p[E]$ is {\itshape bounded}. Crucially, the differentials on $D_{p}[E]$ are bounded and everywhere defined as linear operators, ensuring that both $D_p[E]$ and $D_{p,abs}[E]$ are {\itshape bounded} Hilbert $\vnN(\Gamma)$-cochain complexes. 
\subsection{Comparsion of the spectrum near $0$}
We will be interested in comparing the spectral density functions of the three complexes. First 
\begin{proposition}\label{density} For $0 \leq p \leq m$, let $\Delta_p^{\perp}[E]: \Omega^p_{(2)}(M,E) \to \Omega^p_{(2)}(M,E)$ the orthogonal Laplacian. Suppose that $\Omega^p_{(2)}(M,E)$ has trivial $L^2$-cohomology. Then we have \begin{equation}
F(\Delta_p[E],\sqrt{\lambda})=F(\Delta_p^{\perp}[{E}],\sqrt{\lambda}) = F_p(L_p[{E}],\lambda) + F_{p-1}(L_{p-1}[{E}], \lambda) \end{equation} for all $\lambda \geq 0$. \end{proposition}
\begin{proof} Since there is no $L^2$-cohomology, we have both $\Delta_p[E] = \Delta_p^\perp[E]$, as well as $F_p(L_p[E],\lambda) \stackrel{\text{Definition}}{=} F(d^p|_{\im(d^{p-1})^\perp},\lambda) \stackrel{\ker(d^{p}) = \overline{\im(d^{p-1})}}{=} F((d^p)^\perp,\lambda)$ for each $p$. Using the identities $F(f \oplus g,\lambda) = F(f,\lambda) + F(g,\lambda)$ and $F(f,\lambda) = F(f^*f,\sqrt{\lambda})$, we compute 
\begin{align*} &  F(\Delta_p^{\perp}[{E}],\sqrt{\lambda}) 
 = F\left( ((d^p)^* d^p)^\perp \oplus (d^{p-1} (d^{p-1})^*)^\perp, \sqrt{\lambda} \right) \\ & = F(((d^p)^* d^p)^\perp, \sqrt{\lambda}) + F_p((d^{p-1} (d^{p-1})^*)^\perp,\sqrt{\lambda}) 
 = F((d^p)^\perp,\lambda) + F((d^{p-1})^\perp,\lambda) \\ & = F_p(L_p[{E}],\lambda) + F_{p-1}(L_{p-1}[{E}],\lambda). 
\end{align*} \end{proof} Recall that the vector bundle $E \downarrow M$ under consideration is {\itshape trivial}, i.e, of the form $M \times V \equiv M \times \mathbb C^m$ for some $m \in \mathbb N$.
Triviality of $E \downarrow M$ allows us to identify smooth sections of $E$ with smooth maps from $M$ to $\mathbb C^m$, which leads to an identification of $C^\infty(M,\mathbb C)$-modules \begin{align}  \Omega^p(M,E) \cong \Omega^p(M)^m,  \\ \Omega^p(\partial M, E) \cong \Omega^p(\partial M)^m. \end{align}  Assume that $\partial M \neq \emptyset$. For appropriate $w > 0$, let $\partial M_w \cong [0,w) \times \partial M$ be the geodesic collar around $\partial M$ of width $w$. Under this identification, we therefore find for any $\omega \in \Omega^p_c(M,E)$ appropriate smooth $1$-parameter families of forms $\omega_1(t) \subset \Omega^p_c(\partial M)^m$ and $\omega_2(t) \subset \Omega^{p-1}_c(\partial M)^m$, such that 
\begin{equation} \omega(t,x) = \omega_1(t)(x) + dt \wedge \omega_2(t)(x), \hspace{.5cm} \forall (t,x) \in [0,w) \times \partial M \cong \partial M_w. \end{equation}
Denote by $\hat{\Delta}_p[E]$ the $p$-th Laplacian on the flat Hermitian restriction bundle $E \downarrow \partial M$.  Let $\phi: [0,w] \to \reals^+$ be a smooth map identically $1$ near $0$ and identically $0$ for all $t > w/2$. With this data in mind, define
\begin{align}& K^p: \Omega^p_c(M,E) \to \Omega^{p-1}(M,E), \\
& K^p \omega \coloneqq \begin{cases} \phi(u) \cdot \int_{0}^u e^{-te^{1+\hat{\Delta}_p[E]}} \omega_2(t)(\;.\;) dt & \text{on} \; \partial M_w, \\ 0 & \text{elsewhere}. \end{cases} 
\end{align}
An immediate, but important consequence is that $K^p \omega$ depends only on the restriction $\omega|_{\partial M_w}$ and that the support of $K^p\omega$ lies in $\partial M_w$. Just as in the case with for the trivial bundle, one can now proceed line by line as in \cite[Lemma 5.5, Proposition 5.6]{Lueck:3-mflds} to show the following:
\begin{proposition}\label{chainhtpy} Let $p \in \mathbb N$ and assume that $\Gamma$ is uniform. Then, for $r = 0,1,2$, the map $K^p$ extends to a bounded operator. \begin{equation} 
 K^p_r: \sob_r^p(M,E) \rightarrow  \sob_{r+1}^{p-1}(M,E).  \end{equation} The norm $||K^p_r||$ depends only on $r$ and flat isometry class of the restriction bundle $E|_{\partial M_w} \downarrow \partial M_w$. Setting $K_{-1}^p \coloneqq 0$ for all $p \in \mathbb N$, we furthermore obtain \begin{enumerate}
\item For $* = 0,1,2$ and each $p \in \mathbb N$, the map \begin{equation} j_p^{p+1-*} = 1 - dK_{*}^{p+1-*} + K_{ * - 1}^{p + 2 - *}d: \sob_{*}^{p+1-*}(M,E) \to \sob_{*}^{p+1-*}(M,E) \end{equation} has image in $\sob_{*,abs}^{p+1-*}(M,E)$ and extends to a morphism of Hilbert $\mathcal N(\Gamma)$-cochain complexes $j_p: D_p[E] \to D_{p,abs}[E]$. 
\item The {\itshape reduced complexes} \begin{align*}& \overline{D}_p[E] \coloneqq \dots 0 \rightarrow \sob_{1}^p(E) / \overline{\im(d)} \to \sob_{0}^{p+1}(E) \rightarrow 0 \dots, \\
& \overline{D}_{p,abs}[E] \coloneqq \dots 0 \rightarrow \sob_{1,abs}^p(E) /\overline{\im(d)} \to \sob_{0}^{p+1}(E) \rightarrow 0 \dots \end{align*}
are chain homotopy equivalent. More precisely, the map $j_p$ descends to a map $\overline{j}_p: \overline{D}_p[E] \to \overline{D}_{p,abs}[E]$, that is the chain homotopy inverse to the induced inclusion $\overline{i}_p: \overline{D}_{p,abs}[E] \to \overline{D}_p[E]$. The respective null-homotopies are induced from $i_* \circ K^*_*$ and $K^*_* \circ i_*$. 
\end{enumerate}
\end{proposition}
We now obtain an important, intermediate result, relating the spectral density functions of the three different de Rham type complexes we have defined previously. In order to understand its proof, recall that we have an inclusion of $\Gamma$-invariant subspaces $\dots \subseteq \sob_2^*(M,E) \subseteq \sob_1^*(M,E) \subseteq \sob_0^*(M,E) = \Omega_{(2)}^*(M,E)$. Given $k \in \mathbb N_0$, a subset $A \subseteq \sob_k^*(M,E)$ is said to be {\itshape $j$-closed} for $j \leq k$ if it is closed in the the $\sob_j$-topology. Using the convention $A^{\perp_0} = A^\perp$, we will denote by $A^{\perp_k} \subseteq \sob_k^*(M,E)$ the $\sob_k$-orthogonal complement of $A$ inside $\sob_k^*(M,E)$. 
\begin{proposition}\label{estimate3} We find constants $C_1,C_2 > 0$, depending only on the flat isometry class of the restriction bundle $E|_{\partial M_w} \downarrow \partial M_w$, (with $\partial M_w \coloneqq \emptyset$ if $\partial M = \emptyset$) such that all of the following hold: \begin{enumerate} \item 
If $\partial M \neq \emptyset$ and $\Gamma$ is uniform, we have \begin{equation*} F_p(D_{p,abs}[E],C_1^{-1} \lambda) \leq F_p(D_{p}[E],\lambda) \leq F_p(D_{p,abs}[E],C_1\lambda) \end{equation*} for all $\lambda \leq C_2$. 
\item We have \begin{equation*} F_p(L_p[E],\lambda) \leq F_p(D_{p,abs}[E],\lambda) \leq F_p(L_p[E],\sqrt{2}\lambda) \end{equation*} for all $\lambda \leq \frac{1}{\sqrt{2}}$.
\item If either $\partial M = \emptyset$ or if $\partial M \neq \emptyset$ and $\Gamma$ is uniform, we have  \begin{align*} F_p(L_{p}[E],C_1^{-1}\lambda) \leq  F_p(D_p[E],\lambda) \leq F_p(L_p[E],C_1\sqrt{2}\lambda). \end{align*} for all $\lambda \leq \min\{ C_2, \frac{1}{C_2\sqrt{2}} \}$.
\end{enumerate}
\end{proposition}
\begin{proof}
(1): Follows from Theorem \ref{DILEQ} and Proposition \ref{chainhtpy}. \\
(2): Throughout the proof, we will use the equalities \begin{align}\label{sobequal} \ker( d^p|_{\mathcal W^{1,p}_{abs}(E)})^{\perp_1} = \dom(d^p) \cap \ker(d^p)^{\perp} = \mathcal W^{p}_{1,abs}(E) \cap \overline{\im(\delta^p)}, \end{align} as established in \cite[Lemma 5.9]{Lueck:hyp}, as well as the identity \begin{equation} F(f,\lambda) = \sup \{ \vntr(p_L): L \subseteq \mathcal H \; \text{closed, $\Gamma$-invariant subspace}  \; : ||f(x)|| \leq \lambda ||x|| \; \forall x \in L \}, \end{equation} as established in \cite[Lemma $2.3$]{Lueck:Book}. We will first show that \begin{align*}F_p(L_p[E],\lambda) \leq F_p(D_{p,abs}[E],\lambda). \end{align*} For this, we let $L \subseteq \mathcal D(d^p) \cap \ker(d^p)^{\perp}$ be a $0$-closed, $\Gamma$-invariant subspace, such that each $v \in L$ satisfies $||dv||_0 \leq \lambda ||v||_0$. By the left-hand equality of \ref{sobequal}, and the fact that the $1$-topology is stronger than the $0$-topology, it follows that $L \subseteq \mathcal W^{p}_{1,abs}(E)$ is a $1$-closed, $\Gamma$-invariant subspace. Since $||v||_0 \leq ||v||_1$, we also have $||dv||_0 \leq \lambda ||v||_1$. All in all, this implies that $F_p(L_p[E],\lambda) \leq F_p(D_{p,abs}[E],\lambda) $. \\
In order to show that \begin{align*}F_p(D_{p,abs}[E],\lambda) \leq F_p(L_p[E],\sqrt{2}\lambda), \end{align*} we let $L \subseteq  \ker( d^p|_{\mathcal W^{1,p}_{abs}(E)})^{\perp_1}$ be a $1$-closed, $\Gamma$-invariant subspace, satisfying $||dv||_0 \leq \lambda ||v||_1$ for any $v \in L$. 
Now observe that by the equality of the first and the last term in Equation \ref{sobequal}, we also have $||v||_1^2 = ||v||_0^2 + ||dv||_0^2$ for any $v \in L$, which implies that
\begin{equation} ||v||_0 \leq ||v||_1 \leq \sqrt{1 - \lambda^2} ||v||_0, \end{equation} for any $v \in L$. Therefore, $L$ is also $0$-closed inside $\mathcal D(d^p) \cap \ker(d^p)^{\perp}$. 
Moreover, \begin{align*} &||dv||_0^2 \leq \lambda^2 ||v||_1^2  \leq \lambda^2  ||v||_0^2  +\frac{1}{2} ||dv||_0 \\ & \implies ||dv||_0 \leq \sqrt{2}\lambda||v||_0, \end{align*} from which the inequality $F_p(D_{p,abs}[E],\lambda) \leq F_p(L_p[E],\sqrt{2}\lambda)$, and therefore the desired claim, follows. \\
(3): Keeping in mind that $D_{p,abs}[E] = D_{p}[E]$ whenever $\partial M = \emptyset$, (3) follows immediately from (1) and (2).
\end{proof}

Adopting the notation from the previous section, the ultimate goal of this section is to find a uniform polynomial upper bound for the spectral density functions of the complex $\Omega^*(M_R, E_R)$, {\bfseries independent of $R$}, which will then easily imply the desired large-time convergence. 

\begin{proposition}\label{equiv}  There exists constants $C_1,C_2 > 0$ independent of $R \geq 1$, such that for all $p =0,\dots,m$ and all $\lambda \leq \min\{ C_2, \frac{1}{C_2\sqrt{2}} \}$, we have
\begin{align*} & F_p(L_{p}[E^\rho],C_1^{-1}\lambda) \leq  F_p(D_p[E^\rho],\lambda) \leq F_p(L_p[E^\rho],C_1\sqrt{2}\lambda), \\
& F_p(L_{p}[E^\rho_{M_R}],C_1^{-1}\lambda) \leq  F_p(D_p[E^\rho_{M_R}],\lambda) \leq F_p(L_p[E^\rho_{M_R}],C_1\sqrt{2}\lambda), \\ 
& F_p(L_{p}[E^\rho_{T_R}],C_1^{-1}\lambda) \leq  F_p(D_p[E^\rho_{T_R}],\lambda) \leq F_p(L_p[E^\rho_{T_R}],C_1\sqrt{2}\lambda).
 \end{align*}
\end{proposition}

\begin{proof} 
Since $\Gamma$ acts cocompactly on both $M_R$ and $T_{R}$, we get the above inequalities from Proposition \ref{estimate3}, choosing $w = 1/3$ (the width of the geodesic collar around the boundary), with constants $C_1(R), C_2(R)$, depending {\itshape a priori} on $R > 1$, but only on the flat isometry classes of the restrictions $E^\rho_{\partial (M_R)_{1/3}} \downarrow \partial (M_R)_{1/3}$ and $E^\rho|_{\partial (T_{R})_{1/3}} \downarrow \partial (T_{R})_{1/3}$. Let $\partial(M_R)_{1/3}^0$ be a connected component of $\partial(M_R)_{1/3}$ and let $\partial (T_{R})_{1/3}^0$ be the intersection of $\partial (T_{R})_{1/3}$ with a connected component of $T_R$. Then, from the explicit end structure laid out in the beginning of Section \ref{FLATRHO}, there are isometric diffeomorphisms \begin{align*} & \partial(M_R)_{1/3}^0 \cong [R - 1/3,R] \times \reals^{n-1}, \\ & \partial (T_{R})_{1/3}^0 \cong ([R,R+1/3] \dot \cup [R+2/3,R+1]) \times \reals^{n-1}, \end{align*} each sending the hyperbolic metric to the warped product metric $dt^2 + e^{-2t} dx^2$. Using the very same local isometries and following the same arguments as in Lemma \ref{BUNDISO3}, we obtain two bundle isometries 
\begin{align*} & E^\rho|_{\partial (M_R)_{1/3}} \downarrow \partial (M_R)_{1/3} \cong E^\rho|_{\partial (M_1)_{1/3}} \downarrow \partial (M_1)_{1/3}, \\ &
E^\rho|_{\partial (T_{R})_{1/3}} \downarrow \partial (T_{R})_{1/3} \cong E^\rho|_{\partial (T_{1})_{1/3}} \downarrow \partial (T_{1})_{1/3}. \end{align*}
Consequently, we get $C_1(R) = C_1(1) \eqqcolon C_1$ and $C_2(R) \eqqcolon C_2(1) \eqqcolon C_2$ for any $R > 1$ and the result follows. 
\end{proof}

\begin{proposition}\cite[Theorem $1.1$, Proposition $1.2$]{Durland:Novi}\label{l2nov1} The Hilbert $\vnN(\Gamma)$-cochain complex $\Omega_{(2)}^*(\hyp^n,E^\rho)$ has trivial cohomology and positive Novikov-Shubin invariants. \end{proposition}

We want to show that the same holds true for the Hilbert $\vnN(\Gamma)$-cochain complex $\Omega_{(2)}^*(T_{1}, E^\rho_{T_1})$. 

\begin{proposition}\label{l2nov2} For each $R > 0$, the Hilbert $\vnN(\Gamma)$-cochain complex $\Omega_{(2)}^*(T_{R}, E^\rho_{T_R})$  has trivial cohomology and positive Novikov-Shubin invariants. The same holds true for the complex $\Omega_{(2)}^*(\partial T_R, E^\rho_{\partial T_R})$. \end{proposition}
\begin{proof} Since the proof method does not take into account the specific choice of $R$, we will prove it only for the case $R = 1$. For each $1 \leq j \leq k$, let $T_1^j$ be the complete submanifolds of $T_1$ with $\Gamma_0^j \coloneqq \{\gamma \in \Gamma: \gamma. T_1^j = T_1^j \}$ the stabilizer of $T_1^j$, so that we have a decomposition \begin{equation*} T_1 \cong \coprod_{j=1}^k \Gamma \times_{\Gamma_0^j} T_1^j, \end{equation*} as detailed in Section $2.3$. In the same section, we have shown that $\Gamma_0^j \subset \Isom^+(T_1^j)$ is a uniform lattice, isomorphic to $\mathbb Z^{d-1}$. Therefore, the result follows from the equality of analytic and topological $L^2$-Betti numbers and Novikov-Shubin invariants presented in Corollary \ref{DeRhCor}, followed by the induction principle from Lemma \ref{indprinc} and finally Propositon \ref{novpos1}. For $\partial T_1$, we replace $T_1^j$ by $\partial T_1^j = T_1^j \cap \partial T_1 \cong \reals^{n-1}$ and proceed completely analogously as above (using that $\Gamma_0^j$ restricts to a cocompact lattice on $\partial T_1^j$). 
\end{proof}

Notice that the previous proof does {\bfseries not} relate the different positive Novikov-Shubin invariants from the different complexes $\Omega_{(2)}^*(T_{R}, E^\rho_{T_R})$. {\itshape A priori}, these do depend on $R$. Showing that there exists a positive uniform lower bound requires a little more work:

%In view of Remark \ref{detclassrem}, note that in the course of the previous proof, we have in fact also shown the following:
%\begin{corollary}\label{DETCLASSFORBD} For each $R > 0$, the pair $(\partial M_R, \rho)$ is $L^2$-acyclic and of determinant class. \end{corollary}

\begin{proposition}\label{essent} There exists constants $\epsilon, \alpha > 0$, such that for all $R \geq 1$  and all $0 \leq p \leq m$ the following hold true \begin{enumerate} 
\item  For all $R \geq 1$ we have \begin{equation} F(\Delta_p[E^\rho_{T_R}],\lambda) \leq F(\Delta_p[E^\rho_{T_1}],\lambda). \end{equation}
\item For all $\lambda < \epsilon^{-1}$, we have \begin{align*} F_p(D_p[E^\rho],\lambda) < \epsilon \cdot \lambda^\alpha, \\
F_p(D_p[E^\rho_{T_R}], \lambda) < \epsilon \cdot \lambda^\alpha. \end{align*}
In particular, both $D_p[E^\rho]$ and $D_p[E^\rho_{T_R}]$ are Fredholm at $p$ and have vanishing $p$-th cohomology. 
\end{enumerate}
\end{proposition}

\begin{proof}
\begin{enumerate}
\item There is a flat bundle isometry $F_R: E^\rho_{T_R} \downarrow {T}_R \to E^\rho_{T_1} \downarrow {T}_1$, as defined in Lemma \ref{BUNDISO3}. Consequently, by Equation \ref{equivariance}, we have for any $x \in T_R$ and any $\lambda \geq 0$, that 
\begin{equation} \tr( \chi_{[0,\lambda^2]}(\Delta_p[{E}_0])(F_R(x),F_R(x))) =  \tr(\chi_{[0,\lambda^2]}(\Delta_p[E^\rho_{T_R}])(x,x)). \end{equation}
We may choose fundamental domains $\mathcal D_R \subseteq {T}_R$ and $\mathcal D_1 \subseteq {T}_1$ for the respective $\Gamma$-actions, satisfying $F_R(\mathcal D_R) \subseteq \mathcal D_1$. Therefore, we have 
\begin{align*} & F(\Delta_p[E^\rho_{T_R}],\lambda) = \vntr(\chi_{[0,\lambda^2]}(\Delta_p[E^\rho_{T_R}]) ) 
= \int_{\mathcal D_R} \tr(\chi_{[0,\lambda^2]}(\Delta_p[E^\rho_{T_R}])(x,x))) dx \\ & = \int_{\mathcal F_R(\mathcal D_R)} \tr( \chi_{[0,\lambda^2]}(\Delta_p[E^\rho_{T_1}])(F_R(x),F_R(x))) dx \\&
\leq \int_{\mathcal D_1} \tr( \chi_{[0,\lambda^2]}(\Delta_p[E^\rho_{T_1}])(F_R(x),F_R(x))) dx =  F(\Delta_p[E^\rho_{T_1}],\lambda). \end{align*}
\item Define $\beta \coloneqq \min \{ \alpha(\Delta_p[X]):  0 \leq p \leq m, X \in \{E^\rho,E^\rho_{T_1} \} \}$. Propositions \ref{l2nov1} and \ref{l2nov2} imply that we have $\beta > 0$, as well as both $\Delta_p[E^\rho] = \Delta_p^\perp[E^\rho]$ and $\Delta_p[E^\rho_{T_1}] = \Delta_p^\perp[E^\rho_{T_1}]$. For $E^\rho$, we can apply Propositions \ref{equiv} and \ref{density} to find a constant $c \geq 1$, such that for all $0 \leq p \leq m$, and all $\lambda < c^{-1}$ we have
\begin{equation*}  F_p(D_p[E^\rho],\lambda) \leq F_p(L_p[E^\rho], c\lambda) \leq F(\Delta_p[E^\rho], \sqrt{c \lambda} ) < c^\alpha \lambda^\alpha. \end{equation*}
For $E^\rho_{T_R}$, we can use the same argument, along with assertion $(2)$, similarly yielding 
\begin{align*}& F_p(D_p[E^\rho_{T_R}], \lambda) \leq F_p(L_p[E^\rho_{T_R}], c\lambda) \leq F_p(\Delta_p[E^\rho_{T_R}], \sqrt{c\lambda}) \\
& \leq F_p(\Delta_p[E^\rho_{T_1}], \sqrt{c\lambda}) \leq c^{\alpha/2} \lambda^{\alpha/2}. \end{align*} 
Setting $\epsilon \coloneqq \max \{ c^{\alpha/2},c\}$ and $\beta = \alpha/2 > 0$, we obtain the result.
\end{enumerate}
\end{proof}

For $R > 1$, denote by \begin{align} i_{M_R}: M_R \to \hyp^n, \\ i_{C_R}: C_R \to \hyp^n, \\ i_{({T}_{R},-)}: T_{R} \to M_{R+1}, \\ i_{({T}_{R},+)}: T_{R} \to C_{R}. \end{align}
the respective smooth inclusion maps. Each of these induces a $\Gamma$-invariant map between the corresponding twisted de Rham complexes, bounded with norm $1$. The proof for the scalar-values case of the following important theorem, presented in the reference, carries over to our situation of bundle-valued forms without further modification: 
\begin{lemma}\cite[Lemma 5.14]{Lueck:3-mflds}\label{exactsob} For any $R > 1$ and any $0 \leq p \leq m$, the sequence of morphisms of Hilbert $\vnN(\Gamma)$-cochain complexes \begin{equation} 0 \rightarrow D_p[E^\rho] \xrightarrow{j_R} D_p[E^\rho_{M_R}] \oplus D_p[E^\rho_{C_{R-1}}] \xrightarrow{q_R} D_p[E^\rho_{T_{R-1}}] \rightarrow 0 \end{equation} 
is exact. Here, for smooth forms, we have \begin{align} & j_R\omega \coloneqq i^*_{{M}_R}\omega \oplus i^*_{C_{R-1}} \omega,\\
& q_R(\omega_1 \oplus \omega_2) \coloneqq i^*_{{T}_{(R-1,-)}} \omega_1 - i^*_{{T}_{(R-1,+)}} \omega_2. \end{align}
\end{lemma}

\begin{lemma} 
\label{exactbound} There exists a constant $C > 0$, such that for all $R > 1$ and all $0 \leq p \leq m$, we have
\begin{equation} F_p( D_p[E^\rho_{M_R}],\lambda) \leq F_p(D_p[E^\rho],C \cdot \lambda^{1/4}) + F_p(D_p[E^\rho_{T_{R-1}}],C \cdot \lambda^{1/2}) \hspace{.3cm} \text{for} \; 0 \leq \lambda \leq C^{-1}. \end{equation}
\end{lemma}

\begin{proof} Consider the exact sequence \begin{equation} 0 \rightarrow D_p[E^\rho] \xrightarrow{j_R} D_p[E^\rho_{M_R}] \oplus D_p[E^\rho_{C_{R-1}}] \xrightarrow{q_R} D_p[E^\rho_{T_{R-1}}] \rightarrow 0 \end{equation} from Lemma \ref{exactsob}. Because of Proposition \ref{essent}, the outer two complexes are Fredholm at $p$ and, moreover, have vanishing $p$-cohomology. In particular, we can apply Proposition \ref{exactleq} to obtain that  \begin{equation}
F_p(D_p[E^\rho_{M_R}],\lambda) \leq F_p(D_p[E^\rho],c_1(R) \cdot \lambda ) + F_p(D_p[E^\rho_{T_{R-1}}], c_1(R) \cdot \lambda), \end{equation} for all $\lambda < c_2(R)$, where $c_1(R)$ and $c_2(R)$ are constants given by rational expressions of the norms of $q_R,j_R$, the differential on $D_p[E^\rho_{M_R}] \oplus D_p[E^\rho_{C_{R-1}}]$, and their respective inverses. Using the flat bundle isometries given in Lemma \ref{BUNDISO3}, one can now proceed analogously as in \cite[Lemma 6.6]{Lueck:hyp}  to show that these norms are bounded from above by universal constants independent of $R > 1$, thus proving the lemma. \end{proof}

\begin{proposition}\label{specbound} There exists constants $C,\beta >0$, such that for all $0 \leq p \leq m$, the following hold: \begin{enumerate} 
\item The Hilbert $\vnN(\Gamma)$-cochain complex $\Omega^*_{(2)}(M_R,E^\rho_{M_R})$ has vanishing $p$-th homology. Equivalently, we have \begin{equation} \Delta_p[E^\rho_{M_R}] = \Delta_p^\perp[E^\rho_{M_R}]. \end{equation}
\item For all $R > 1$, we have a uniform bound on the spectral density functions as follows
\begin{equation} F_p(\Delta_p[E_{M_R}], \lambda) \leq C  \lambda^{\beta}  \hspace{.3cm} \text{for} \; 0 \leq \lambda \leq C^{-1}. \end{equation}
\end{enumerate}
\end{proposition}

\begin{proof} First, observe that the $p$-th cohomology of $\Omega^*(M_R,E^\rho_{M_R})$ is isomorphic to the $p$-th cohomology of $L_p[E^\rho_{M_R}]$. To show that the latter is trivial, we only need to show that \begin{equation}\label{l2triv} F_p(L_p[E^\rho_{M_R}],0) = 0. \end{equation} Using Lemma \ref{exactbound}, together with both Proposition \ref{equiv} and Proposition \ref{density}, we obtain constants $\alpha,c,C_1,C_2 > 0$ independent of $R > 1$, such that 
\begin{align*} & F_p( D_p[E^\rho_{M_R}],\lambda) \leq F_p(D_p[E^\rho],c \lambda^{1/4}) + F_p(D_p[E^\rho_{T_{R-1}}],c  \lambda^{1/2}) \\
& \leq F_p(L_p[E^\rho],cC_1\sqrt{2}\lambda^{1/4}) + F_p(L_p[E^\rho_{T_{R-1}}],cC_1\sqrt{2}\lambda^{1/2}) \\ &
\leq F_p(\Delta_p^{\perp}[E^\rho],\sqrt{cC_1\sqrt{2}}\lambda^{1/8}) + F_p(\Delta_p^{\perp}[E^\rho_{T_{R-1}}],\sqrt{cC_1\sqrt{2}}\lambda^{1/4}). \\
& \leq 2 (cC_1\sqrt{2})^{\alpha/2} \lambda^{\alpha/8}  \end{align*} for $\lambda \leq \min \{1,c^{-1}, C_1^{-1}, C_2, \frac{1}{\sqrt{2}C_2} \}$ and all $0 \leq p \leq m$. In particular, since $F(L_p[E^\rho_{M_R}],0) \leq F_p(D_p[E^\rho_{M_R}],0)$ by Proposition \ref{estimate3}, Equation \ref{l2triv} immediately follows from the above computation. \\  Applying now again Proposition \ref{density} and Proposition \ref{equiv} together with the above inequality, we compute
\begin{align*} & F_p(\Delta_p[E^\rho_{M_R}],\lambda) =  F_p(\Delta_p^\perp[E^\rho_{M_R}],\lambda) \leq F_p(D_p[E^\rho_{M_R}],C_1\lambda^2) + F_{p-1}(D_{p-1}[E^\rho_{M_R}],C_1\lambda^2)  \\ &
\leq 4 (cC_1^{3/2}\sqrt{2})^{\alpha/2} \lambda^{\alpha/4} \end{align*} for $\lambda \leq \min \{1, c^{-1} ,C_1^{-1},C_2,\frac{1}{\sqrt{2}C_2}, C_1^{-1}C_2, \frac{1}{C_1C_2\sqrt{2}} \} \coloneqq C'$. \\ Setting $C \coloneqq \max \{ (C')^{-1}, 4(cC_1^{3/2}\sqrt{2})^{\alpha/2} \}$ and $\beta \coloneqq \alpha/4$, we obtain the desired result.
\end{proof} 

Neatly summarized, we get the following result:

\begin{corollary}\label{bddetclass} For each $R > 0$, the pair $(M_R,\rho|_{M_R})$ is $L^2$-acyclic and of determinant class. In fact, there exists a uniform constant $\beta > 0$ {\bfseries independent of $R$}, so that for each $0 \leq p \leq n$, one has $\alpha_p^{An}(M_R,\rho) \geq \beta$. \end{corollary}

\subsection{Proof of large-time convergence}

For the next result, choose a nested sequence \begin{equation} \dots \subset \mathcal F_{R-1} \subset \mathcal F_R \subset F_{R+1} \subset \dots \end{equation} inside $\hyp^n$,  where $\mathcal F_R$ is a compact fundamental domain for the $\Gamma$-action on $M_R$ and 
$\mathcal F \coloneqq \bigcup_{R > 1} \mathcal F_R$ is a finite-volume fundamental domain for the $\Gamma$-action on $\hyp^n$. Thus, we have in particular \begin{align} \label{volest1}\Vol(\mathcal F_R) < \Vol(\mathcal F) < \infty, \\\label{volest2} \lim_{R \to \infty} \Vol(\mathcal F \setminus \mathcal F_R) = 0. \end{align}

\begin{proposition}\label{heatconv} For all $t \geq 1$, we have \begin{equation} \lim_{R \to \infty} \vntr(e^{-t\Delta_p^\perp[E^\rho_{M_R}]}) = \vntr(e^{-t\Delta_p^\perp[E^\rho]}). \end{equation} \end{proposition}

\begin{proof}
Proposititons \ref{l2nov1} and  \ref{specbound} tell us that both $\Delta_p^\perp[E^\rho_{M_R}] = \Delta_p[E^\rho_{M_R}]$ and $\Delta_p^\perp[E^\rho] = \Delta_p[E^\rho]$. Moreover, since $R \geq 2$, we may apply Theorem \ref{heat} to find appropriate constants $c,C_1,C_2 > 0$ independent of $R$ and $t$, so that \begin{equation*}  |\tr(\heatR)(x,x) - \tr(\heatH)(x,x)|  \leq C_1e^{-\frac{R^2}{C_2t}} \end{equation*} for all $x \in \mathcal F_{R/2}$. Together with \ref{volest1} and \ref{volest2}, this implies that \begin{align*} & | \vntr(\heatR) - \vntr(\heatH) | = | \int_{\mathcal F_R} \tr(\heatR)(x,x) dx - \int_{\mathcal F} \tr(\heatH)(x,x) | \\ & \leq 
\int_{\mathcal F_{R/2}} |\tr(\heatR)(x,x) - \tr(\heatH)(x,x)| dx + \int_{\mathcal F_R - \mathcal  F_{R/2}}   |\tr(\heatR)(x,x)| dx  \\ & + \int_{\mathcal F - \mathcal F_{R/2}} |\tr(\heatH)(x,x)| dx  
 \leq \Vol(\mathcal F_{R/2}) C_1 e^{-\frac{R^2}{C_2t}} + \Vol(\mathcal F \setminus \mathcal F_{R/2}) c. \xrightarrow{R \to \infty} 0. \end{align*} 
\end{proof}

\begin{lemma}\label{heatbound} There exists a positive function $G \in C^0(1,\infty) \cap L^1(1,\infty)$, such that for all $t \geq 1$ and all sufficiently large $R >> 0$, we have \begin{equation} t^{-1}\vntr(e^{-t\Delta_p^{\perp}[E_{M_R}]}) \leq G(t). \end{equation}\end{lemma}
\begin{proof} Throughout the proof, we will abbreviate \begin{equation} F_R^p(\lambda) \coloneqq F_p(\Delta_p[E^\rho_{M_R}],\lambda). \end{equation} We compute that one has for any $R > 0$, any $t \geq 1$ and arbitrary $\epsilon > 0$
\begin{align*}& \vntr(e^{-t\Delta_p^\perp}[E_{M_R}]) = \int_{0}^\infty e^{-t\lambda} dF_R^p(\lambda)   = \int_{0}^\epsilon e^{-t\lambda} dF_R^p(\lambda) + \int_{\epsilon}^\infty e^{-t\lambda} dF_R(\lambda) \\& \stackrel{F_R^p(0) = 0}{=} t \int_{0}^\epsilon e^{-t\lambda} F_R^p(\lambda) d\lambda + e^{-t\epsilon} F_R^p(\epsilon) +  \int_{\epsilon}^\infty e^{-t\lambda} dF_R(\lambda) \\& \stackrel{t > 1}{\leq}  t \int_{0}^\epsilon e^{-t\lambda} F_R^p(\lambda) d\lambda + e^{-t\epsilon} F_R^p(\epsilon) +  e^{-t\epsilon}\int_{\epsilon}^\infty e^{-\lambda+\epsilon} dF_R(\lambda) \\& = t \int_{0}^\epsilon e^{-t\lambda} F_R^p(\lambda) d\lambda + e^{-t\epsilon} F_R^p(\epsilon) +  e^{-t\epsilon}e^\epsilon \vntr(e^{-\Delta_p^\perp[E_{M_R}]}), \end{align*} and therefore in particular \begin{equation} t^{-1}\vntr(e^{-t\Delta_p^{\perp}[E^\rho_{M_R}]}) \leq \int_{0}^\epsilon e^{-t\lambda} F_R^p(\lambda) d\lambda+ \frac{e^{-t\epsilon}}{t}(F_R^p(\epsilon) + e^{\epsilon} \vntr(e^{-\Delta_p^{\perp}[{E^\rho_{M_R}]}})). \end{equation}
By Proposition \ref{heatconv} we find some $\delta > 0$, such that for all $R >> 0$, we have $\vntr(e^{-\Delta_p^{\perp}[E^\rho_{M_R}]}) \leq \vntr(e^{-\Delta^{\perp}_p[E^\rho]}) + \delta$. Similarly, by Proposition \ref{specbound}, we can choose $\epsilon, \beta > 0$ independently of $R$ and small enough, such that $F_R^p(\lambda) \leq \epsilon^{-1} \lambda^{\beta}$ for all $\lambda \leq \epsilon$. From this, it becomes obvious that the function
\begin{equation} G(t) \coloneqq \epsilon^{-1} \int_{0}^\epsilon e^{-t\lambda} \lambda^{\beta} d\lambda + \frac{e^{-t\epsilon}}{t}( \epsilon^{-1+\beta} + e^{\epsilon} ( \vntr(e^{-\Delta^{\perp}_p[E^\rho]}) + \delta) ) \end{equation} satisfies the assertions of our lemma. 
\end{proof}

Using Proposition \ref{heatconv}, Lemma \ref{heatbound} and Lebesgue's theorem of dominated convergence, we finally obtain the main result of this section:

\begin{theorem}[Large-time convergence]\label{LTC} We have \begin{equation} \lim_{R \to \infty} \int_{1}^\infty t^{-1} \vntr(\heatR) dt = \int_{1}^\infty t^{-1} \vntr(\heatH) dt. \end{equation} In particular, we obtain 
\begin{equation} \lim_{R \to \infty} \sum_{p=0}^n \int_{1}^\infty t^{-1} \vntr(\heatR) dt = \sum_{p=0}^n\int_{1}^\infty t^{-1} \vntr(\heatH) dt.  \end{equation} \end{theorem}

\bibliographystyle{plain}
\bibliography{BibTeXCM}{}

\begin{bibdiv}
\begin{biblist}

\bib{Ash:Tor}{article}{
      author={A.~Ash, M. McConnell D.~Yasaki, P. E.~Gunnells},
       title={On the growth of torsion in the cohomology of arithmetic groups},
        date={2020},
     journal={J. Inst. Math. Jussieu},
      volume={19},
      number={2},
       pages={537\ndash 569},
}

\bib{Petronio:hyp}{book}{
      author={Benedetti, R.},
      author={Petronio, C.},
       title={Lectures on {H}yperbolic {G}eometry},
      series={Universitext},
   publisher={Springer},
        date={1992},
}

\bib{Bergeron:Growth}{article}{
      author={Bergeron, N.},
      author={Venkatesh, A.},
       title={The asymptotic growth of torsion homology for arithmetic groups},
        date={2013},
     journal={J. Inst. Math. Jussieu},
      volume={12},
      number={2},
       pages={391\ndash 447},
}

\bib{Friedlander:Bd}{article}{
      author={Burghelea, D.},
      author={Friedlander, L.},
      author={Kappeler, T.},
       title={Torsions for manifolds with boundary and glueing formulas},
        date={1999},
     journal={Math. Nachr.},
      volume={208},
       pages={31\ndash 91},
}

\bib{Friedlander:Uni}{article}{
      author={Burghelea, D.},
      author={Friedlander, L.},
      author={Kappeler, T.},
      author={Macdonald, P.},
       title={{A}nalytic and {R}eidemeister torsion for representations in
  finite type {H}ilbert modules},
        date={1996},
     journal={Geom. Funct. Anal.},
      volume={6},
      number={5},
       pages={751\ndash 859},
}

\bib{Durland:Novi}{unpublished}{
      author={Durland, P.},
       title={Twisted {L$^2$}-{T}orsion of {H}yperbolic {M}anifolds},
        date={2013},
        note={Master's {T}hesis, {M}athematisch-{N}aturwissenschaftliche
  {F}akult\"at der {R}heinischen {F}riedrich-{W}ilhelms-{U}niversit\"at
  {B}onn},
}

\bib{Farrell:Wh}{article}{
      author={Farrell, T.},
      author={Jones, L.~E.},
       title={Rigidity for aspherical manifolds with $\pi_1 \subset
  \text{GL}_m({R})$},
        date={1998},
     journal={Asian J. Math.},
      volume={2},
      number={2},
       pages={215\ndash 262},
}

\bib{Gromov:Novi}{article}{
      author={Gromov, M.},
      author={Shubin, M.~A.},
       title={Von {N}eumann spectra near zero},
        date={1991},
     journal={Geom. Funct. Anal.},
      volume={1},
      number={4},
       pages={375\ndash 404},
}

\bib{Kammeyer:L2}{book}{
      author={Kammeyer, H.},
       title={Introduction to {L}$^2$-invariants},
      series={Lecture Notes in Mathematics, 2247},
   publisher={Springer, Cham},
        date={2019},
}

\bib{Kato:Func}{book}{
      author={Kato, T.},
       title={Pertubation theory for linear operators},
      series={Grundlehren der mathematischen {W}issenschaften},
   publisher={Springer},
        date={1984},
      volume={132},
}

\bib{Lueck:Novi}{article}{
      author={L{\"u}ck, W.},
       title={Estimates for spectral density functions of matrices over
  $\ceals[{Z}^d]$},
        date={2015},
     journal={Ann. Math. Blaise Pascal},
      volume={22},
      number={1},
       pages={73\ndash 88},
}

\bib{Lück:twist}{article}{
      author={L{\"u}ck, W.},
       title={Twisting {L$^2$}-invariants with finite-dimensional
  representations},
        date={2018},
     journal={J. Topol. Anal.},
      volume={10},
      number={4},
       pages={723\ndash 816},
}

\bib{Lueck:Book}{book}{
      author={L\"{u}ck, W.},
       title={{L$^2$}-invariants: {T}heory and applications to geometry and
  $k$-theory},
      series={Ergebnisse der Mathematik und ihrer Grenzgebiete. 3. Folge},
   publisher={Springer-Verlag, Berlin},
        date={2002},
      volume={44},
        ISBN={3-540-43566-2},
}

\bib{Lueck:3-mflds}{article}{
      author={L{\"u}ck, W.},
      author={Lott, J.},
       title={{L$^2$}-topological invariants of $3$-manifolds},
        date={1995},
     journal={Invent. {M}ath.},
      volume={120},
      number={1},
       pages={15\ndash 60},
}

\bib{Lueck:hyp}{article}{
      author={L{\"u}ck, W.},
      author={Schick, T.},
       title={{L$^2$}-torsion of hyperbolic manifolds of finite volume},
        date={1999},
     journal={Geom. Func. Anal.},
      volume={9},
      number={3},
       pages={518\ndash 567},
}

\bib{Matsushima:Metric}{article}{
      author={Matsushima, Y.},
      author={Murakami, S.},
       title={{O}n {V}ector {B}undle {V}alued {H}armonic {F}orms and
  {A}utomorphic {F}orms on {S}ymmetric {R}iemannian {M}anifolds.},
        date={1963},
     journal={Ann. of Math. (2)},
      volume={78},
       pages={365\ndash 416},
}

\bib{Muller:Tor2}{article}{
      author={M\"uller, W.},
       title={Analytic torsion and $\emph{R}$-torsion for unimodular
  representations},
        date={1993},
     journal={J. Amer. Math. Soc.},
      volume={6},
      number={3},
       pages={721\ndash 753},
}

\bib{Muller:heat}{article}{
      author={M\"uller, W.},
       title={Analytic torsion of complete hyperbolic manifolds of finite
  volume},
        date={2012},
     journal={J. Funct. Anal.},
      volume={262},
      number={9},
       pages={2615\ndash 2675},
}

\bib{Muller:Hyp}{article}{
      author={M\"uller, W.},
      author={Rochon, F.},
       title={Analytic torsion and {R}eidemeister torsion of hyperbolic
  manifolds with cusps},
        date={2020},
     journal={Geom. Funct. Anal.},
      volume={30},
      number={3},
       pages={910\ndash 954},
}

\bib{Raim:Tor}{article}{
      author={Raimbault, J.},
       title={Exponential growth of torsion in abelian coverings},
        date={2012},
     journal={Algebr. Geom. Topol.},
      volume={12},
      number={3},
       pages={1331\ndash 1372},
}

\bib{Reid:Hyp}{book}{
      author={Reid, M.},
      author={Balazs, S.},
       title={Geometry and topology},
   publisher={Cambridge {U}niversity {P}ress},
        date={2005},
        ISBN={978-0-521-84889-3},
}

\bib{Scholze:Tor}{article}{
      author={Scholze, P.},
       title={On torsion in the cohomology of locally symmetric varieties},
        date={2015},
     journal={Ann. of Math. (2)},
      volume={182},
      number={3},
       pages={945\ndash 1066},
}

\bib{Thang:Tor}{article}{
      author={Thang, L.},
       title={Homology torsion growth and mahler measure},
        date={2014},
     journal={Comment. Math. Helv.},
      volume={89},
      number={3},
       pages={719\ndash 757},
}

\bib{Ich}{thesis}{
      author={Wa{\ss}ermann, B.},
       title={The {L$^2$}-{C}heeger-{M}\"uller theorem for representations of
  hyperbolic lattices},
        type={Ph.D. Thesis},
        date={2020},
        note={{K}arlsruhe {I}nstitute for {T}echnology},
}

\bib{Ich2}{article}{
      author={Wa{\ss}ermann, B.},
       title={An {L$^2$}-{C}heeger-{M}\"uller theorem on compact manifolds with
  boudary},
        date={2020},
        note={Preprint, https://arxiv.org/abs/2004.08367},
}

\bib{Zhang:CM}{article}{
      author={Zhang, W.},
       title={An extended {C}heeger-{M}\"uller theorem for covering spaces},
        date={2005},
     journal={Topology},
      volume={44},
      number={6},
       pages={1093\ndash 1131},
}

\end{biblist}
\end{bibdiv}

\end{document}